\documentclass[12pt]{article}
\usepackage[margin=1.0in]{geometry}
\usepackage[utf8]{inputenc}
\usepackage{tikz}
\usepackage{mathtools}
\usepackage{mathrsfs,amssymb,amsfonts,amsthm,amsmath,amsbsy}
\usepackage[english]{babel}
\usepackage[T1]{fontenc}
\usepackage{hyperref}
\usepackage{float}
\usepackage{upgreek}
\usepackage{enumitem}
\usepackage{bm}
\usepackage{amssymb}
\usepackage{mdframed}
\usepackage{verbatim}
\usepackage{dsfont}
\usepackage{stmaryrd}
\usepackage{soul}
\usepackage{xcolor}
\usepackage{titlesec}
\usepackage{listings}
\usepackage{lipsum}
\usepackage[export]{adjustbox}
\usepackage[
backend=biber,
style=numeric,maxnames=99
]{biblatex}
\renewbibmacro{in:}{}
\newtheorem{theorem}{Theorem}[section]
\newtheorem{proposition}[theorem]{Proposition}
\newtheorem{corollary}[theorem]{Corollary}

\newtheorem{lemma}[theorem]{Lemma}

\newtheorem{definition1}[theorem]{Definition}

\newtheorem{remark1}[theorem]{Remark}

\newcommand{\somme}[3]{\sum_{#1}^{#2} #3}
\newcommand{\indi}[1]{\mathds{1}_{#1}}
\newcommand{\Bn}{B^{+}(n)}
\newcommand{\Hate}{\mathcal{H}_{a}(t,e)}
\newcommand{\Jate}{J_{a}(t,e)}
\newcommand{\Jatec}{\overset{\circ}{J_a}(t,e)}
\newcommand{\Bnb}{B^{\bullet}(n)}
\newcommand{\Srn}{\mathcal{S}^{+}_r(n)}
\newcommand{\Srb}{\mathcal{S}^{\bullet}_r(n)}
\newcommand{\AnII}{\mathcal{A}_n(\varepsilon,\eta,\beta)}

\newcommand{\ct}{w_{\theta}}
\newcommand{\Tng}{T_{n,g_n}}
\newcommand{\Tngp}{T_{n,g_n,\mathbf{p}^n}}
\newcommand{\en}{e^n}

\author{Tanguy Lions\footnote{ENS Lyon, tanguy.lions@ens-lyon.fr}}

\date{}
\title{\textbf{Typical distances in high-genus triangulations}}

\graphicspath{{./figures/}{./}}
\begin{document}
\maketitle
\begin{abstract}
We study the distance between two uniformly chosen points on a uniform random triangulation whose genus $g$ is proportional to the number of faces $2n$. We show that the distance rescaled by $\log(n)$ converges in probability to a deterministic constant, which answers a conjecture of Budzinski, Chapuy and Louf \cite{budzinski2023distancesisoperimetricinequalitiesrandom}. The proof relies on the precise study of the volume growth of the ball of radius $r$ for $r$ of order $\log(n)$. The main ingredients are the recent local convergence results for uniform triangulations with boundaries obtained in \cite{lions2026locallimitsuniformtriangulations} and the isoperimetric inequalities obtained in \cite{budzinski2023distancesisoperimetricinequalitiesrandom}.
 \end{abstract}

\section{Introduction}

\paragraph{High genus random geometry.}
Random structures on high-genus surfaces have attracted a lot of interest from the mathematical community recently. Their geometric properties are difficult to study and are typically expected to exhibit expander-like behaviour, just like those of random graphs \cite{BOLLOBAS1988241,Bollobas_diamater,RIORDAN_2010}. 

A model that has attracted the attention of a large community of geometers (and, more recently, probabilists) is that of random hyperbolic surfaces sampled according to the Weil--Petersson measure. In the high-genus regime $g \to +\infty$, Mirzakhani proved that, with high probability, random hyperbolic surfaces exhibit expander-like behaviour and have logarithmic diameter \cite{mirzakhani2010growthweilpeterssonvolumesrandom}. Very recent works show that the spectral gap for typical random high-genus hyperbolic surfaces is concentrated around the maximal possible value $\frac{1}{4}$ \cite{anantharaman2025friedmanramanujanfunctionsrandomhyperbolic,anantharaman2026friedmanramanujanfunctionsrandomhyperbolic}.  Finally, the convergence of the lengths of the shortest closed geodesics has also been obtained \cite{Mirzakhani_2019,budd2026tightlengthspectrumlargegenus}. More precise estimates for distances can be obtained for models of random hyperbolic surfaces that are built using random graphs \cite{mathien2026diameternewmodelrandom,Budzinski_2021_diameter,Belyi,mathien2026sharperboundminimalpossible}. In many cases, the quantities of interest coincide with the analogue on the universal cover (the Poincaré disk).

In the discrete setting, random maps are the natural analogue of random surfaces. Unicellular maps (i.e., maps with a single face) have been shown to contain a large expander with high probability \cite{Louf2021LargeEI}. Writing $n$ for the number of vertices, their typical distances and diameter are known to be of order $\log(n)$ \cite{diameter_unicellular}. The convergence of the lengths of the shortest geodesics has been proven \cite{length_spectrum_unicellular,length_spectrum_unicellular2} and shows surprisingly similar behaviour compared to the Weil--Petersson model. Beyond the unicellular case, considering random uniform high-genus triangulations, it was recently established in \cite{lions2026largeexpandersubgraphshigh} that, with high probability, these maps contain a large expander. In \cite{budzinski2023distancesisoperimetricinequalitiesrandom}, denoting by $n$ the number of faces and by $g$ the genus, as soon as $g/n \to \theta \in (0,1/2)$, it was shown that the typical graph distances and diameter of uniform high-genus triangulations are of logarithmic order with high probability. The different results described above on distances give the order of magnitude of distances but they do not provide exact asymptotic equivalents. In the present paper, we build upon these foundations to refine the estimates for typical distances in uniform high-genus triangulations, proving that they concentrate around $D_{\theta}\log(n)$ for some deterministic constant $D_{\theta} > 0$.

\paragraph{Main theorem.}
We now define the finite models of interest. Fix integers $g \ge 0$ and $n \ge 2g-1$. We define the set of triangulations $\mathcal{T}(n,g)$ as the set of maps $t$ with genus $g$ and $2n$ faces of degree $3$, equipped with a distinguished oriented edge called the root edge. See Section~\ref{preliminaries} for more precise definitions. We denote by $T_{n,g}$ an element of $\mathcal{T}(n,g)$ chosen uniformly at random. We set $\lambda_c = \frac{1}{12 \sqrt{3}}$. For any $\lambda \in (0,\lambda_c]$, let $h = h(\lambda)$ be the unique $h \in (0,\frac{1}{4}]$ such that $\lambda = \frac{h}{(1+8h)^{\frac{3}{2}}}$ and let
\begin{align*}
d(\lambda) = \frac{h \log\Big(\frac{1+\sqrt{1-4h}}{1-\sqrt{1-4h}}\Big)}{(1+8h)\sqrt{1-4h}}.
\end{align*} It is shown in \cite{Budzinski_2020} that $d(\lambda)$ is increasing with $d(\lambda_c) = \frac{1}{6}$ and $\lim_{\lambda \to 0}d(\lambda) = 0$. Thus, for any $\theta \in [0,\frac{1}{2})$, we define $\lambda(\theta)$ as the unique number $\lambda \in (0,\lambda_c]$ such that $$d(\lambda(\theta)) = \frac{1-2\theta}{6}.$$ We also introduce $h_{\theta} := h(\lambda(\theta))$ and we set $m_{\theta} = \frac{1-2h_{\theta}-\sqrt{1-4h_{\theta}}}{2h_{\theta}} \in (0,1]$. We write $d_{T_{n,g}}$ for the graph distance on $T_{n,g}$. We can now state our main theorem.

\begin{theorem}\label{theorem_typical_distances}
     Fix $\theta \in (0,\frac{1}{2})$ and $\displaystyle \frac{g_n}{n} \to \theta$ as $n \to +\infty$. Let $e_n^1,e_n^2$ be two independent oriented edges chosen uniformly at random on $\Tng$. Let $x_n^1$ and $x_n^2$ denote their respective starting points. Then, we have
    \begin{align}\label{convergence_proba}
       \frac{d_{\Tng}(x_n^1,x_n^2)}{\log(n)} \xrightarrow[n \to +\infty]{(\mathbb{P})} D_{\theta},
    \end{align}
where $\displaystyle D_{\theta} = \frac{1}{\log(m_{\theta}^{-1})}$. The same result holds if $x_n^1,x_n^2$ are chosen to be independent uniform vertices on $\Tng$.
    \end{theorem}
We briefly discuss the origin of the constant $D_{\theta}$. In \cite{Budzinski_2020}, the authors show that $\Tng$ converges locally in distribution to a random triangulation of the plane called PSHT defined in \cite{PSHT}. Moreover, the volume of the ball of radius $r$ in the PSHT is typically of order $\exp(\frac{r}{D_{\theta}})$. This explains why the constant $D_{\theta}$ appears in the distances in $\Tng$ (see the strategy below).

Let us now compare Theorem~\ref{theorem_typical_distances} with the existing literature on high-genus random geometry. All the results mentioned above fall into one of the two following situations:

\begin{enumerate}
    \item \textbf{Order-of-magnitude estimates:} Expander properties (such as logarithmic diameter or typical distances, Cheeger constant estimates, spectral gap) are established only up to a constant factor \cite{budzinski2023distancesisoperimetricinequalitiesrandom, mirzakhani2010growthweilpeterssonvolumesrandom, diameter_unicellular}.
    \item \textbf{Simple local structure:} The model has a simple local structure like the hyperbolic plane \cite{anantharaman2025friedmanramanujanfunctionsrandomhyperbolic, anantharaman2026friedmanramanujanfunctionsrandomhyperbolic, Budzinski_2021_diameter} or an infinite tree \cite{local_unicellular, Bollobas_diamater, RIORDAN_2010} in discrete settings. Comparing to this structure yields easy 'a priori' bounds which sometimes happen to be sharp.
\end{enumerate}

To our knowledge, Theorem~\ref{theorem_typical_distances} provides the first sharp asymptotic for typical distances in a setting where the local limit has a rich structure. Indeed, it captures the exact constant despite a significantly more intricate local geometry than that of the aforementioned models.

    \paragraph{Strategy of the proof.}
We give an overview of the general strategy used to prove Theorem~\ref{theorem_typical_distances}. Our first observation is that
\begin{center}
$d_{\Tng}(x_n^1,x_n^2) \le r \iff x_n^2 \in B_{r}(\Tng,x_n^1)$,
\end{center}
where $B_{r}(\Tng,x_n^1)$ denotes the ball of radius $r$ centered at $x_n^1$ (see Section~\ref{local_topology}) (we will also use the notation $B_r^+$ for a related notion defined precisely in Section~\ref{section_hulls}). It follows that we can write
\begin{align*}
\mathbb{P}\big(d_{\Tng}(x_n^1,x_n^2) \le r \text{ | }x_n^1,\Tng\big)
= \frac{|B_{r}(\Tng,x_n^1)|}{6n},
\end{align*}

where $6n$ denotes the total number of oriented edges in $\Tng$. Thus, the main task is to understand the growth of $(B_r(\Tng,e_n))_{r \ge 0}$, where $e_n$ is a uniformly chosen oriented edge on $\Tng$. We now explain how this is done. We list below the main ingredients used throughout the paper:
\begin{itemize}
	\item[$\bullet$]\textbf{Local limits.} 
	First, using \cite[Theorem~1]{Budzinski_2020}, the sequence of triangulations $(\Tng,e_n)_{n \ge 0}$ converges locally in distribution to an infinite triangulation of the plane called PSHT. A direct consequence is that for any $r \ge 0$ fixed, the volume of the ball of radius $r$ of $(\Tng,e_n)$ converges in distribution to the volume of the ball of radius $r$ of the PSHT. Since the latter grows as $(m_{\theta}^{-1})^r$, we deduce that for any fixed $r \ge 0$, we have for $n$ large enough $|B_r(\Tng,e_n)| \approx (m_{\theta}^{-1})^r$. Unfortunately, this argument works only for bounded values of $r$ and we need to extend the previous result to values of $r$ of order $\log(n)$. For this, we want to prove that $|B_{r+1}(\Tng,e_n)| \approx m_{\theta}^{-1} |B_r(\Tng,e_n)|$ as long as $|B_r(\Tng,e_n)| = o(n)$. The key ingredient is a recent local limit for uniform high-genus triangulations with boundaries \cite[Theorem 1.1]{lions2026locallimitsuniformtriangulations}. The idea is as follows: assume that we have discovered $B_r(\Tng,e_n)$, that $|B_r(\Tng,e_n)| = o(n)$ and that $\Tng \setminus B_r(\Tng,e_n)$ is\footnote{In general, there is no reason for $\Tng \setminus B_r(\Tng,e_n)$ to be connected. As we will see in Section~\ref{section_hulls}, we will consider another notion of balls (the hulls) where this is true.} connected, and observe that $\Tng \setminus B_r(\Tng,e_n)$ is a uniform high-genus triangulation with boundaries. To understand $B_{r+1}(\Tng,e_n)$, one needs to understand the local behaviour of $\Tng \setminus B_r(\Tng,e_n)$ close to the boundary faces (see Figure~\ref{strategy1}). This is exactly done in \cite[Theorem 1.1]{lions2026locallimitsuniformtriangulations} which states that, in distribution, the component $\Tng \setminus B_r(\Tng,e_n)$ rooted on a boundary converges locally in distribution to the half-plane analogue of the PSHT. A consequence of that is the desired estimate $|B_{r+1}(\Tng,e_n)| \approx m_{\theta}^{-1} |B_r(\Tng,e_n)|$. By iterating this procedure, we expect that $|B_{r}(\Tng,e_n)| \approx m_{\theta}^{-r}$ as long as $|B_r(\Tng,e_n)| = o(n)$.

\begin{figure}[H]
\centering
\includegraphics[scale=0.2]{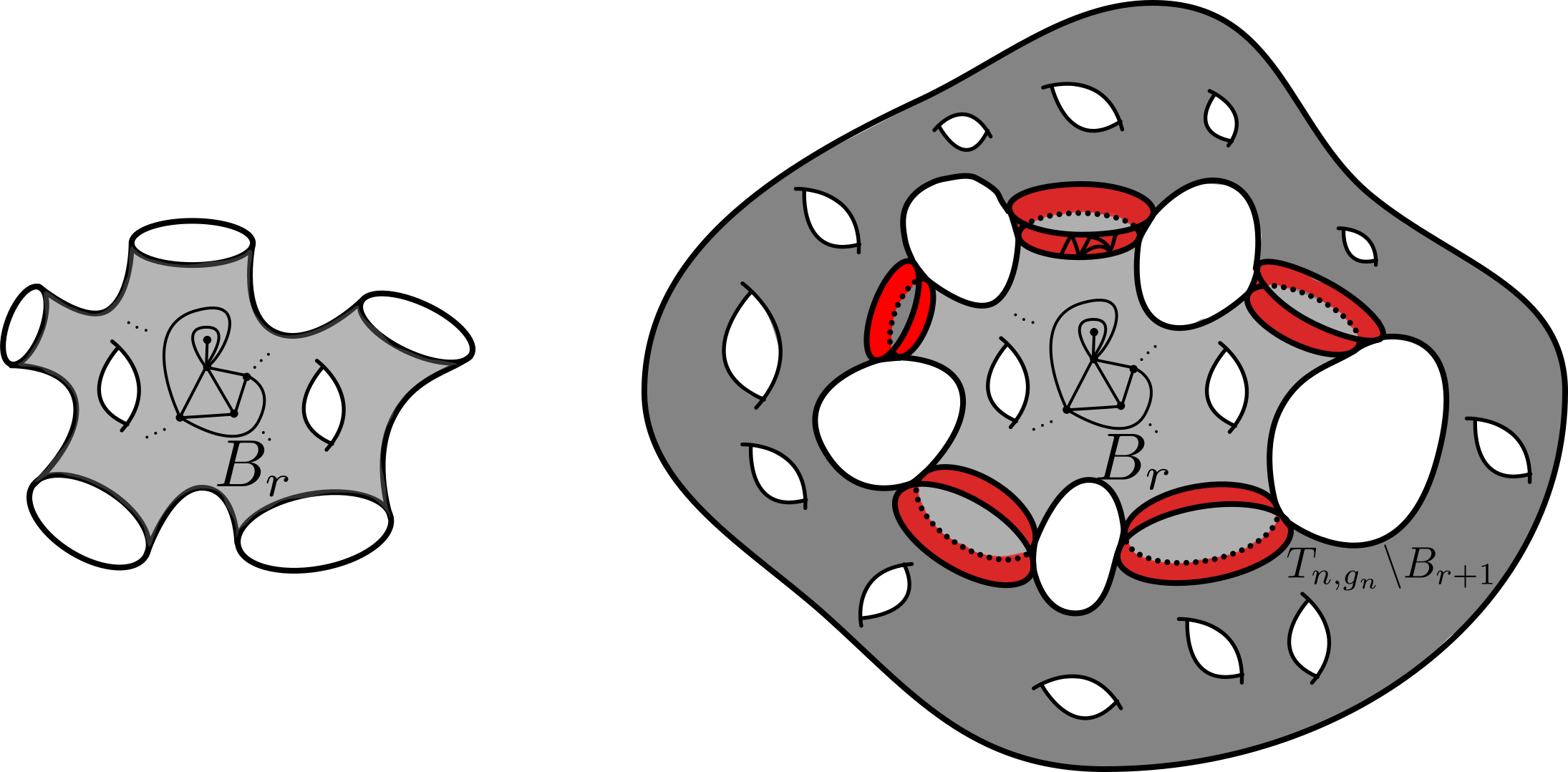}
\caption{On the left, we represent the ball of radius $r$, denoted by $B_r$. On the right, we represent $\Tng$. In this example, $\Tng \setminus B_r$ consists of a single connected component. The red region corresponds to $B_{r+1} \setminus B_r$. Thus, the red region can be understood using the local limit of $\Tng \setminus B_r$ along its boundary.}
\label{strategy1}
\end{figure}

	However, three issues arise. First, it may happen that $\Tng \setminus B_r(\Tng,e_n)$ does not consist of a single connected component (see Figure~\ref{strategy2}). For example, it could happen that $\Tng\setminus B_r(\Tng,e_n)$ has two components $T_1,T_2$ of sizes $n_1,n_2$ and genera $g_1,g_2$, with $\frac{g_1}{n_1}$ not close to $\frac{g}{n}$, which prevents us from applying \cite{lions2026locallimitsuniformtriangulations}. Second, even though for any $r \ge 0$ such that $|B_r(\Tng,e_n)| = o(n)$, with high probability, the ratio $\frac{|B_{r+1}(\Tng,e_n)|}{|B_r(\Tng,e_n)|}$ is close to the predicted value, it does not imply that with high probability this holds simultaneously for all $r \ge 0$ such that $|B_r(\Tng,e_n)| = o(n)$.
	 The third issue is that our strategy relies on \cite{lions2026locallimitsuniformtriangulations} which only applies as long as $|B_r(\Tng,e_n)| = o(n)$.

	\item[$\bullet$]\textbf{Isoperimetric inequalities and hulls.}
	We address the first and third issues in this bullet, and defer the second one to the next bullet. The first and third problems can both be resolved using the isoperimetric inequalities established in \cite{budzinski2023distancesisoperimetricinequalitiesrandom}. Informally (see Section~\ref{section_isoperimetric_inequalities} for more precise definitions), they state that, with high probability, for any subset $X \subset \Tng$, as long as $|X|$ is of order at least $\log(n)$ and $|X| \le \frac{n}{2}$, we have $|\partial X| \ge c |X|$ for some universal constant $c > 0$. A consequence is that as long as $|B_r(\Tng,e_n)| = o(n)$ and $|B_r(\Tng,e_n)|$ is of order at least $\log(n)$, the ball $B_r(\Tng,e_n)$ might disconnect $\Tng$ into several connected components, but only one of these components contains almost all the mass and almost all the genus of $\Tng$ (see Figure~\ref{strategy2}).

\begin{figure}[H]
\centering
\includegraphics[scale=0.15]{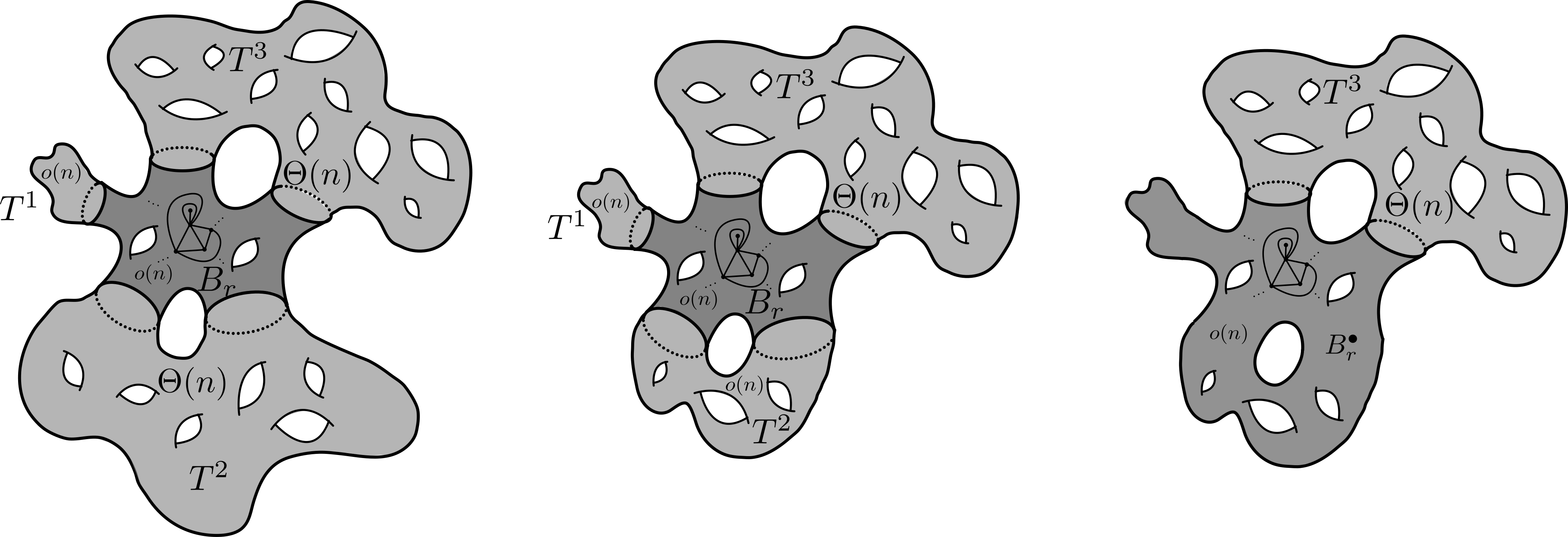}
\caption{In these three figures, we assume that $ |B_r(\Tng,x_n)| = o(n)$ and $|B_r(\Tng,x_n)|$ is of order at least $\log(n)$. On the left, we represent a pathological situation where both $T^2$ and $T^3$ contain a positive proportion of the total mass of $\Tng$. By the isoperimetric inequality, this event occurs with small probability. In the middle, the component $T^3$ contains almost all the mass of $\Tng$. This is the most likely situation. On the right, the hull of radius $r$ is shown in dark grey and is obtained by gluing the small components $T^1$ and $T^2$ to $B_r$.}
\label{strategy2}
\end{figure}

	Therefore, it is more convenient to study the \emph{hull}, denoted by $B_r^{\bullet}(\Tng,x_n)$, obtained by gluing the connected components except the "large" one to $B_r(\Tng,x_n)$. This solves the first problem since by definition of the hull $\Tng \setminus B_r^{\bullet}(\Tng,x_n)$ is connected. Moreover, by the isoperimetric inequality, we will argue that $|B_r^{\bullet}(\Tng,x_n)|$ approximates $|B_r(\Tng,x_n)|$ well.

	For the third problem, the previous sketch of proof shows that $|B_{(1-\varepsilon)D_{\theta}\log(n)}^{\bullet}(\Tng,e_n)| \approx n^{1-\varepsilon}$. To fill the gap, we will simply use the isoperimetric inequality which gives the crude bound $|B_{r+1}^{\bullet}| \ge (1+c)|B_r^{\bullet}|$ with $c > 0$. The constant $c$ is \emph{a priori} not the optimal constant for the growth but sufficient for our needs here since $\varepsilon$ is small.

	\item[$\bullet$]\textbf{Second moment bounds for the degree distribution.}
	To address the second problem, we must ensure that on the small probability event where the local convergence of \cite{lions2026locallimitsuniformtriangulations} differs from the expected behaviour, the volume of the hull does not grow too fast. To this end, we establish a uniform second moment bound (see Proposition~\ref{second_moment_bounded_degree}) for the degree distribution of the root vertex in random high-genus triangulations with boundaries.
\end{itemize}

\paragraph*{Structure of the paper}
The paper is organized as follows. In Section~\ref{preliminaries}, we introduce all the definitions and notions that we will use throughout the paper. In Section~\ref{section_final_proof}, we give the proof of Theorem~\ref{theorem_typical_distances} assuming Theorem~\ref{growth_of_balls} which gives the correct growth rate for the ball of radius $(1-\varepsilon)D_{\theta}\log(n)$ for any $\varepsilon > 0$. In Section~\ref{section_boundary}, we study high-genus triangulations with boundaries and more precisely their growth around the boundary. Finally, in Section~\ref{section_growth_of_hulls}, we prove Theorem~\ref{growth_of_balls}.

\paragraph{Acknowledgements.} 
I thank Thomas Budzinski for his crucial help at various stages of this project. I also thank Grégory Miermont for stimulating discussions about this project.
\tableofcontents

\section{Preliminaries}\label{preliminaries}
\subsection{Definitions}\label{Definitions}
We begin by recalling the main definitions used throughout the paper.

A (finite or infinite) \emph{map} $m$ is obtained by gluing together a collection of oriented polygons along their edges, with matching orientations, so that the resulting surface is connected. If finitely many polygons are glued, the resulting surface is orientable, and we can define its genus. We say that a map $m$ is \emph{rooted} if it is equipped with a distinguished oriented edge called the \emph{root edge}. The face on the right of the root edge is called the \emph{root face}, and the vertex at its starting point is called the \emph{root vertex}. We denote by $m^{*}$ the dual of $m$, defined as the map where the vertices correspond to the faces of $m$ and two vertices of $m^{*}$ are connected by an edge if the corresponding faces are connected by an edge in $m$.

A \emph{triangulation} is a map all of whose faces have degree~$3$. We focus on \emph{type-I triangulations}, that is, triangulations that may contain multiple edges and loops. For $n \ge 1$ and $g \ge 0$, we denote by $\mathcal{T}(n,g)$ the set of rooted triangulations of genus~$g$ with $2n$ faces. By Euler's formula, a triangulation in $\mathcal{T}(n,g)$ has $3n$ edges and $n+2-2g$ vertices. In particular, the set $\mathcal{T}(n,g)$ is non-empty if and only if $n \ge 2g-1$. We denote by $\tau(n,g)$ the cardinality of $\mathcal{T}(n,g)$, and we write $T_{n,g}$ for a uniform random triangulation in $\mathcal{T}(n,g)$.

In this paper, we will consider maps with boundaries. We introduce two notions, the first including the second.

\begin{definition1}\label{triangulation_with_holes}
A triangulation with holes is a finite map $t$ with:
\begin{itemize}
	\item[$\bullet$] A family of distinguished faces, called the boundaries (or external faces) $\partial_1,\dots,\partial_{\ell}$, of degrees $p_1,\dots,p_{\ell} \ge 1$, which are vertex-simple (that is, all vertices of $\partial_i$ are distinct) and share no vertices. Each boundary $\partial_i$ is equipped with a root edge $e_i$ such that $\partial_i$ lies to the right of $e_i$.
	\item[$\bullet$] A family of distinguished faces $h_1,\dots,h_{r}$ of degrees $q_1,\dots,q_r \ge 1$, called the holes, which are edge-simple (each $h_i$ is composed of $q_i$ edges) and share no edges.
	\item[$\bullet$] Faces that are neither boundaries nor holes have degree~$3$.
	\item[$\bullet$] The dual $t^{*}$ remains connected after removing the boundaries.
\end{itemize} 
The perimeter of $t$ is the tuple $\mathbf{p} =(p_1,\dots,p_{\ell})$, and the boundary length or total perimeter is $|\mathbf{p}| = p_1+\cdots+p_{\ell}$. Note that in the above definition, a boundary and a hole may share vertices and edges. The internal faces of $t$ are the faces that are neither boundaries nor holes. We denote by $F_{\mathrm{in}}(t)$ the set of internal faces of $t$. The internal vertices (resp. edges) are the vertices (resp. edges) that do not lie on a hole.
\end{definition1}

This definition extends naturally to infinite triangulations and also to the case where $p_i = +\infty$. For a triangulation with holes $t$, we write $\partial^{*}t$ for the union of holes and $\partial t$ for the union of boundaries. We denote by $|t| := \#F_{\mathrm{in}}(t)$ the number of internal faces in $t$.

\begin{definition1}\label{triangulation_multipolygon}
For $\ell \ge 0$ and $\mathbf{p} = (p_1,\dots,p_{\ell})\in (\mathbb{N}^{*})^{\ell}$, a triangulation of the $\mathbf{p}$-multi-polygon (or of the $\mathbf{p}$-gon) is a triangulation with holes that has perimeter $\mathbf{p}$ and has no holes.
\end{definition1}

One can think of a triangulation with holes as a neighbourhood of the root in a triangulation of the $\mathbf{p}$-gon (see Figure~\ref{triangulation_multi_polygone}). The holes are faces that can be filled by other triangulations of multi-polygons.

\begin{figure}[H]
    \centering
    \includegraphics[scale = 0.14]{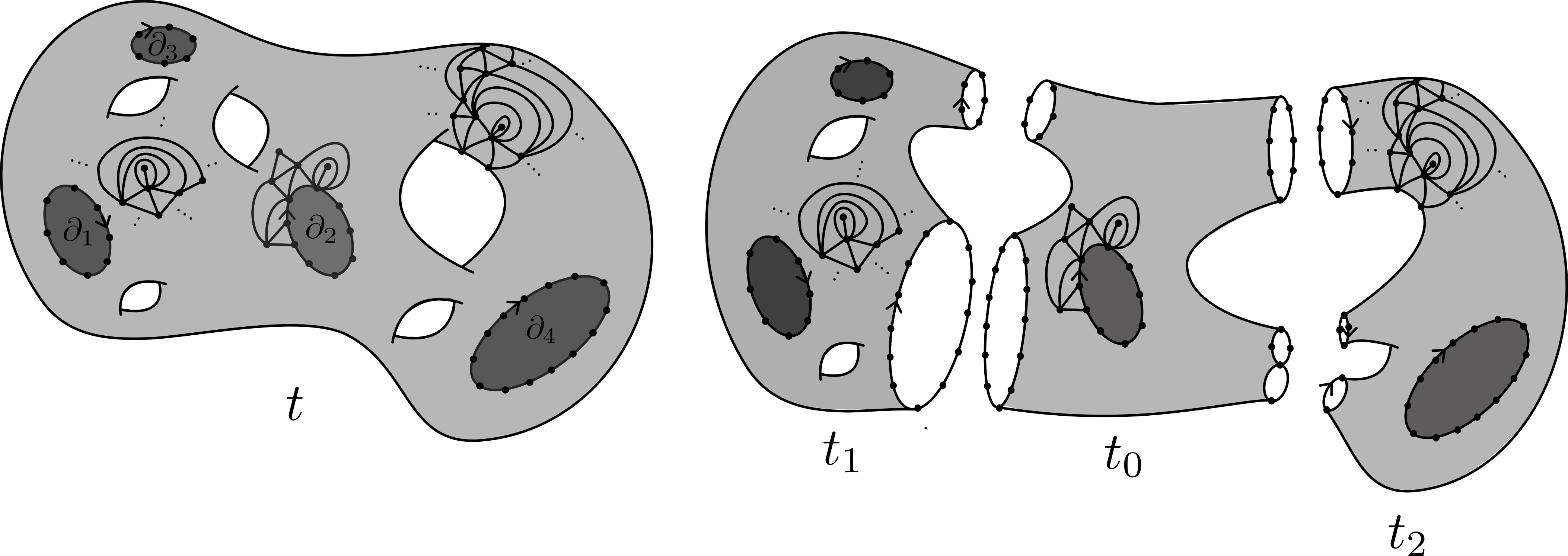}
    \caption{Left: a triangulation of the $(6,8,9,14)$-gon of genus~$5$. 
    Right: a triangulation with holes $t_0 \subset t$, which has $1$ boundary and $5$ holes. Triangulations of multi-polygons $t_1$ and $t_2$ are the connected components of $t \setminus  t_0$. The triangulation $t_1$ (resp. $t_2$) is a triangulation of the $(6,8,5,13)$-gon (resp. $(14,2,4,7)$-gon).}
    \label{triangulation_multi_polygone}
\end{figure}
For $\mathbf{p}=(p_1,\cdots,p_\ell)\in(\mathbb{N}^*)^\ell$, we denote by $\mathcal{T}_{\mathbf{p}}(n,g)$ the set of triangulations of the $\mathbf{p}$-gon of genus $g$ with $2n-\somme{i=1}{\ell}{(p_i-2)}$ internal triangles, and by $\tau_{\mathbf{p}}(n,g)$ its cardinality. One can verify using Euler's formula that a triangulation of the $\mathbf{p}$-gon has $n+2-2g$ vertices.\\
In the following, we want to introduce a notion of local distance (see Section~\ref{local_topology}) between maps. Thus it is convenient to have only one root edge, not as in $\mathcal{T}_{\mathbf{p}}(n,g)$. For $t \in \mathcal{T}_{\mathbf{p}}(n,g)$ and an oriented edge $e$ of $t$, we define the \emph{rooted triangulation of the $\mathbf{p}$-gon} $(t,e)$ and call $e$ its root. We write $\mathcal{T}^1_{\mathbf{p}}(n,g)$ for the set of all these rooted triangulations of multi-polygons. We also define $\mathcal{T}^{1} = \displaystyle \bigsqcup_{\substack{g,n,\ell\ge 0\\p_1,\cdots,p_{\ell}\ge 1}}\mathcal{T}^1_{p_1,\cdots,p_{\ell}}(n,g)$. 

If $(t_0,e_0)$ is a rooted triangulation with holes and $(t,e)$ is a rooted triangulation of the $\mathbf{p}$-gon, we write $(t_0,e_0)\subset (t,e)$ if $(t,e)$ can be obtained\footnote{ Note that by the last item in Definition~\ref{triangulation_with_holes}, if $(t_0,e_0)\subset (t,e)$, there is a unique way to obtain $(t,e)$ by filling the holes of $(t_0,e_0)$} from $(t_0,e_0)$ by gluing one or several triangulations $t_1,\dots,t_k$ of multi-polygons to the holes of $t_0$ (see Figure~\ref{triangulation_multi_polygone}). We write $(t,e) \setminus  (t_0,e_0)$ for the collection of triangulations of multi-polygons that have to be glued to $t_0$ to obtain $(t,e)$.

If $(t,e)$ is an infinite rooted triangulation of the $\mathbf{p}$-gon, we say that it is \emph{one-ended} if, for every $(t_0,e_0) \subset (t,e)$ with a finite number of internal faces, only one connected component of $(t,e)\setminus  (t_0,e_0)$ is infinite. We say that $(t,e)$ is \emph{planar} if, for every $(t_0,e_0) \subset (t,e)$ with a finite number of internal faces, the triangulation with holes $(t_0,e_0)$ is planar. For two vertices $x,y$ in a triangulation with holes $t$, we write $d_{t}(x,y)$ for the graph distance between $x$ and $y$ in $t$. If $e$ is an oriented edge in $t$ with starting point $x$, we write $d_{t}(e,y) = d_{t}(x,y)$. If $e,e'$ are oriented edges in $t$ with starting points $x$ and $x'$, we write $d_{t}(e,e') = d_{t}(x,x')$.

\subsection{Combinatorics}\label{Section_combi}
In this subsection, we recall some combinatorial estimates in the planar case.
 For $p \ge 1$ and $\lambda > 0$ we define $w_{\lambda}(p) = \somme{n \ge 0}{}{\tau_p(n-2+p,0)\lambda^{n}}$. Note that $\tau_p(n-2+p,0)$ denotes the number of triangulations of the $p$-gon with $n$ internal vertices. This quantity is finite if and only if $\lambda \le \lambda_c$. For any $0 < \lambda \le \lambda_c$, we write $h \in (0,\frac{1}{4}]$ as the unique solution to the equation $\lambda = \displaystyle \frac{h}{(1+8h)^{\frac{3}{2}}}$. Then, using \cite{krikun2007explicitenumerationtriangulationsmultiple} one can extract the formula 
 \begin{align*}
 w_{\lambda}(1) = \frac{1}{2} - \frac{1+2h}{2\sqrt{1+8h}},
 \end{align*}
 and for $p \ge 2$
 \begin{align}\label{formula_wlambda}
 w_{\lambda}(p) =  (2+16h)^p \frac{(2p-5)!!}{p!}\frac{((1-4h)p+6h)}{4(1+8h)^{\frac{3}{2}}}.
 \end{align}
Then \cite[$(4)$]{krikun2007explicitenumerationtriangulationsmultiple}
 gives the explicit formula 
 \begin{align}\label{formula_W}
 W_{\lambda}(x) = \frac{\lambda}{2}\bigg(\bigg(1-\frac{1+8h}{h}x\bigg)\sqrt{1-4(1+8h)x}-1+\frac{x}{\lambda}\bigg).
 \end{align} 
 For $p \ge 1$ and $\lambda \in (0,\lambda_c]$, we can assign to any planar triangulation $t$ of the $p$-gon the weight $w_{\lambda}(p,t) = \lambda^{|t|_{\mathrm{in}}}$, where $|t|_{\mathrm{in}}$ denotes the number of internal vertices of $t$. Then we call the Boltzmann triangulation of the $p$-gon (with parameter $\lambda$), the random planar triangulation of the $p$-gon $T^{(p)}_{\lambda}$ that takes value $t$ with probability $\displaystyle \frac{w_{\lambda}(p,t)}{w_{\lambda}(p)}$. We also give some combinatorial estimates from \cite{lions2026locallimitsuniformtriangulations} that will be useful for the rest of the paper. 
\begin{lemma}[{\cite[Lemma 3.2]{lions2026locallimitsuniformtriangulations}}]\label{bound_separating}
For any $n,g\ge 0$ such that $n \ge 2g-1$ we have 
\begin{align*}
    \somme{\substack{n_1+n_2=n\\g_1+g_2 = g}}{}{\tau(n_1,g_1)\tau(n_2,g_2)} \le \tau(n+1,g).
\end{align*}

\end{lemma}
The next Lemma shows that the volumes $\tau_{\mathbf{p}}(n,g)$, $\tau_{\mathbf{p}}(n+1,g)$, $\tau_{(p_1+1,p_2,\cdots,p_{\ell})}(n,g)$ differ from each other by a constant factor as long as $\frac{g}{n}$ is not too close to $\frac{1}{2}$ and $|\mathbf{p}| \ll n$.
\begin{lemma}[{\cite[Lemma 3.3]{lions2026locallimitsuniformtriangulations}}]\label{bounded_ratio_vertices}
    Let $\varepsilon > 0$. Then there is a constant $C_{\varepsilon} > 0$ such that for any $n,g \ge 0$ satisfying $\frac{g}{n}\le \frac{1}{2}-4\varepsilon$ and for every $p_1,\cdots,p_{\ell} \ge 1$ such that $|\mathbf{p}| \le \varepsilon n$, we have 
    \begin{align*}
        1 \le \frac{\tau_{\mathbf{p}}(n+1,g)}{\tau_{\mathbf{p}}(n,g)} \le C_{\varepsilon}  \hspace*{0.2cm}\text{ and }   \hspace*{0.2cm}   C_\varepsilon^{-1} \le \frac{\tau_{(p_1+1,p_2,\cdots,p_{\ell})}(n,g)}{\tau_{(p_1,p_2,\cdots,p_{\ell})}(n,g)} \le 1.
    \end{align*}
\end{lemma}
The next Lemma allows us to compare volumes $\tau_{(p_1,\cdots,p_{\ell})}(n,g)$ with volumes $\tau(n,g)$ where we remove boundary components. 
\begin{lemma}\label{add_new_boundary}
    For any $n,g \ge 0$ such that $n \ge 2g -1$ and any $\mathbf{p} = (p_1,\cdots,p_{\ell}) \in \mathbb{N}^{\ell}$, we have 
    \begin{align*}
        \tau_{(p_1,\cdots,p_{\ell})}(n,g) \le (6n)^{\ell-1}\tau(n,g).
    \end{align*}
\end{lemma}
 \subsection{Isoperimetric inequalities}\label{section_isoperimetric_inequalities}

We introduce the notion of \emph{multicurve}. Given a triangulation $t$, a multicurve is a family $(\gamma_1,\ldots,\gamma_{\ell})$ of simple cycles satisfying the following properties:
\begin{itemize}
\item[$\bullet$] the cycles $\gamma_i$ are edge-disjoint;
\item[$\bullet$] no two cycles cross each other, that is, for any vertex $v$ and any four distinct edges $e_1,e_2,e_3,e_4$ incident to $v$ in this cyclic order, there do not exist two cycles $\gamma_i,\gamma_j$ such that $\gamma_i$ uses $e_1,e_3$ and $\gamma_j$ uses $e_2,e_4$.
\end{itemize}
Furthermore, for $k_1+k_2=2n$, we say that the multicurve is \emph{$(k_1,k_2)$-separating} if the complement
$t \setminus \{\gamma_1 \cup \cdots \cup \gamma_{\ell}\}$
consists of two face-connected components with respectively $k_1$ and $k_2$ triangles.

We now state an isoperimetric inequality proved in \cite{budzinski2023distancesisoperimetricinequalitiesrandom}.

\begin{theorem}\label{isoperimetric_inequalities}
Let $\frac{g_n}{n} \to \theta \in (0,\frac{1}{2})$. There exist constants $K_{\theta},\ct>0$ such that, with high probability, for all $K_{\theta}\log(n) \le k_1 \le k_2$ satisfying $k_1+k_2=2n$, the map $\Tng$ does not contain a multicurve of total length $L \le \ct k_1$ that is $(k_1,k_2)$-separating.
\end{theorem}

For the rest of the paper, we fix constants $K_{\theta},\ct$ given by the above theorem. In words, the previous theorem states that if a multicurve separates the surface into two connected components of size larger than $K_{\theta}\log(n)$, then the total length of the multicurve is of the same order as the size of the smallest connected component.

However, it is shown in \cite{budzinski2023distancesisoperimetricinequalitiesrandom} that, with high probability, there exist regions of size of order $\log(n)$ that are separated from the rest of the surface by a multicurve of total length $o(\log(n))$. Nevertheless, the total mass of these regions is negligible. We now make this statement more precise.

\begin{definition1}
Let $\beta > 0$. Let $g,n > 0$ such that $n \ge 2g-1$. Let $t \in \mathcal{T}(n,g)$ and let $f$ be a face of $t$. We say that $f$ is \emph{$\beta$-isolated} if there exists a $(k_1,k_2)$-separating multicurve $\gamma$ in $t$ such that:
\begin{itemize}
\item[$\bullet$] the face $f$ belongs to the connected component of size $k_1$,
\item[$\bullet$] we have $k_2 \ge \sqrt{n}$,
\item[$\bullet$] the total length $L$ of $\gamma$ satisfies $L \le \beta \min(k_1,k_2)$.
\end{itemize}
\end{definition1}

Let $\mathrm{Isol}_{\beta}(t)$ denote the number of $\beta$-isolated faces in $t$. We then have the following result proved in \cite{budzinski2023distancesisoperimetricinequalitiesrandom}.

\begin{proposition}\label{few_isolated_faces}
Let $\frac{g_n}{n} \to \theta \in (0,\frac{1}{2})$. For any $\eta > 0$, there exists $\beta > 0$ such that, for $n$ large enough,
\begin{align*}
\mathbb{P}\big(\mathrm{Isol}_{\beta}(\Tng) \ge \eta n\big) \le \eta.
\end{align*}
\end{proposition}

To conclude this subsection, we state a proposition that will be useful in the remainder of the paper. This proposition is necessary to ensure that the complement of large subsets of $\Tng$ has \textquotedbl large boundaries\textquotedbl\ which allows one to apply \cite[Theorem 1.1]{lions2026locallimitsuniformtriangulations}

\begin{proposition}\label{multicurves_few_boundaries}
Let $\frac{g_n}{n} \to \theta \in (0,\frac{1}{2})$ and let $\varepsilon > 0$. With high probability, any multicurve $(\gamma_1,\ldots,\gamma_s)$ that is $(k_1,k_2)$-separating with $k_1 \le k_2$, $k_1+k_2=2n$ and
$K_{\theta}\log(n) < k_1 \le n^{1-\varepsilon}$
satisfies
\[
s \le \frac{k_1}{\log(\log(n))}.
\]
\end{proposition}

\begin{proof}
The proof is similar to \cite[Lemma 8]{budzinski2023distancesisoperimetricinequalitiesrandom}.
Let $N_n$ denote the number of multicurves $(\gamma_1,\ldots,\gamma_s)$ that are $(k_1,k_2)$-separating with $k_1 \le k_2$ and $k_1+k_2=2n$ and $K_{\theta}\log(n) < k_1 \le n^{1-\varepsilon}$ and satisfying $s > \frac{k_1}{\log(\log(n))}$.

 Let us fix $k_1 \le k_2$ with $k_1+k_2=2n$ and $K_{\theta}\log(n) < k_1 \le n^{1-\varepsilon}$ and $s > \frac{k_1}{\log(\log(n))}$, and $\vec{\ell}=(\ell_1,\ldots,\ell_s)\in (\mathbb{N}^{*})^s$ such that $|\vec{\ell}|:=\ell_1+\cdots+\ell_s \le 3k_1$. We denote by $X(k_1,\vec{\ell})$ the set of pairs $(t,(\gamma_1,\ldots,\gamma_s))$, such that:
\begin{itemize}
\item[$\bullet$] $t \in \mathcal{T}(n,g_n)$.
\item[$\bullet$] $(\gamma_1,\ldots,\gamma_s)$ is $(k_1,k_2)$-separating and with length vector $\vec{\ell}$.
\end{itemize}

We then have
\begin{align}\label{expectation_N}
\mathbb{E}[N_n]
= \tau(n,g_n)^{-1}
\sum_{k_1,s,\vec{\ell}} \#X(k_1,\vec{\ell}).
\end{align}

Fix $k_1,s$ and $\vec{\ell}$, and let us bound $\tau(n,g_n)^{-1}\#X(k_1,\vec{\ell})$. For $(t,(\gamma_1,\ldots,\gamma_s)) \in X(k_1,\vec{\ell})$, let $e$ denote the root edge of $t$. Cutting along $(\gamma_1,\ldots,\gamma_s)$ produces a pair $(t_1,t_2)$ of triangulations with boundaries of length vector $\vec{\ell}$ and $e$ lies in $t_1$ or in $t_2$. We define
\[
\varphi(t,\gamma_1,\ldots,\gamma_s) = (t_1,t_2,e).
\]
The mapping $\varphi$ is injective. Let us introduce $n_i = \frac{1}{2}(k_i + |\vec{\ell}|-2s)$. Recalling that an element of $\mathcal{T}_{\vec{\ell}}(m,h)$ has $2m-|\vec{\ell}|+2s$ internal triangles, we obtain
\begin{align}\label{bound_card_X}
\tau(n,g_n)^{-1}\#X(k_1,\vec{\ell})
\le 6n\,\tau(n,g_n)^{-1}
\sum_{\substack{g_1+g_2=g_n-s+1}}
\tau_{\vec{\ell}}(n_1,g_1)\tau_{\vec{\ell}}(n_2,g_2).
\end{align}

We have $n_1+n_2 \le n+|\vec{\ell}| \le 2n$ and
$n_1 \le \frac{n^{1-\varepsilon}+|\vec{\ell}|}{2} \le 2n^{1-\varepsilon}$. Using Lemma~\ref{add_new_boundary} and then Lemma~\ref{bound_separating} in a very crude way, we bound the right-hand side by
\begin{align*}
&6n \sum_{g_1+g_2=g_n-s+1} (6n_1)^{s-1}(6n_2)^{s-1}\tau(n_1,g_1)\tau(n_2,g_2)\,\tau(n,g_n)^{-1}\\
&\le 6n(12n^{1-\varepsilon})^{s-1}(12n)^{s-1}
\frac{\tau(n+|\vec{\ell}|-2s+1,g_n-s+1)}{\tau(n,g_n)}\\
&\le 6n(12n^{1-\varepsilon})^{s-1}(12n)^{s-1}
\frac{\tau(n+|\vec{\ell}|+1,g_n-s+1)}{\tau(n,g_n)}.
\end{align*}

The Goulden--Jackson recursion formula \cite[Theorem~4]{GOULDEN2008932} yields
\begin{align*}
\tau(n,g)
= \frac{4}{n+1}\bigg(
n(3n-2)(3n-4)\tau(n-2,g-1)
+ \sum_{\substack{i+j=n-2\\ h+k=g}}
(3i+2)(3j+2)\tau(i,h)\tau(j,k)
\bigg),
\end{align*}
from which we deduce that, for $n\ge 2$,
\[
\tau(n,g) \ge n^2 \tau(n-2,g-1).
\]

It follows that
\begin{align*}
\tau(n,g_n)^{-1}\#X(k_1,\vec{\ell})
\le 6n(12n^{1-\varepsilon})^{s-1}(12n)^{s-1}
n^{-2(s-1)}
\frac{\tau(n+|\vec{\ell}|+2s-1,g_n)}{\tau(n,g_n)}.
\end{align*}

Finally, using Lemma~\ref{bounded_ratio_vertices}, together with
$\frac{k_1}{\log(\log(n))} < s \le |\vec{\ell}| \le 3k_1$
and $k_1 \ge K_{\theta}\log(n)$, we obtain
\begin{align*}
\tau(n,g_n)^{-1}\#X(k_1,\vec{\ell})
&\le \exp(C_{\theta}k_1)\,n^{-\varepsilon(s-1)+1}\\
&\le \exp(C_{\theta,\varepsilon}k_1)\,
n^{-\varepsilon \frac{k_1}{\log(\log(n))}}\\
&= \exp\bigg(
C_{\theta,\varepsilon}k_1
-\varepsilon\frac{\log(n)}{\log(\log(n))}k_1
\bigg),
\end{align*}
where we used the fact that
$n^{1+\varepsilon}
= \exp((1+\varepsilon)\log(n))
\le \exp\!\left(\frac{1+\varepsilon}{K_{\theta}}k_1\right)$.
Substituting this bound into \eqref{expectation_N}, we obtain
\begin{align*}
\mathbb{E}[N_n]
\le \sum_{k_1,k_2,s,\vec{\ell}}
\exp\bigg(
C_{\theta,\varepsilon}k_1
-\varepsilon\frac{\log(n)}{\log(\log(n))}k_1
\bigg).
\end{align*}

Fixing $k_1$, the number of $\vec\ell=(\ell_1,\ldots,\ell_s) \in (\mathbb{N}^{*})^s$ with $|\vec\ell|\le 3k_1$ is at most $2^{3k_1} = 8^{k_1}$. Consequently,
\begin{align*}
\mathbb{E}[N_n]
&\le \sum_{k_1=K_{\theta}\log(n)}^{n^{1-\varepsilon}}
8^{k_1}
\exp\bigg(
C_{\theta,\varepsilon}k_1
-\varepsilon\frac{\log(n)}{\log(\log(n))}k_1
\bigg)
\xrightarrow[n\to +\infty]{ } 0.
\end{align*}
This concludes the proof by a first moment argument.
\end{proof}

For any $\varepsilon,\eta,\beta > 0$, let $\mathbf{Iso}_n(\varepsilon,\eta,\beta)$ denote the event on which the conclusion of Theorem~\ref{isoperimetric_inequalities} holds, $\mathrm{Isol}_{\beta}(\Tng) \le \eta n$, and the conclusion of Proposition~\ref{multicurves_few_boundaries} holds. We write $\mathbf{Iso}_n(\varepsilon)$ for the event on which only Theorem~\ref{isoperimetric_inequalities} and Proposition~\ref{multicurves_few_boundaries} hold. Using Theorem~\ref{isoperimetric_inequalities}, Proposition~\ref{few_isolated_faces}, and Proposition~\ref{multicurves_few_boundaries}, we obtain the following corollary.

\begin{corollary}\label{Iso_typical}
Let $\varepsilon,\eta > 0$. There exists $\beta > 0$ such that, for $n$ large enough,
\begin{align*}
\mathbb{P}(\mathbf{Iso}_n(\varepsilon,\eta,\beta)) \ge 1-\eta.
\end{align*}
Moreover $
\mathbb{P}(\mathbf{Iso}_n(\varepsilon)) \underset{n \to +\infty}{\rightarrow} 1$.
\end{corollary}

\subsection{Local topology}\label{local_topology}

In this subsection, we introduce different notions of balls of radius $r$. Although these notions are related, they both have drawbacks. We start by defining the strong ball of radius $r$. We give a definition that is slightly more general. Indeed, we define the slightly more general notion of balls centered at a family of oriented edges.

\begin{definition1}\label{strong_ball}
Fix $n,g \ge 0$ such that $n \ge 2g-1$. Also fix $\ell \ge 0$ and $\mathbf{p} = (p_1,\ldots,p_{\ell}) \in (\mathbb{N}^{*})^{\ell}$. Given $t \in \mathcal{T}_{\mathbf{p}}(n,g)$ and a subset $E$ of oriented edges of $t$, for any $r \ge 0$ we define the strong ball $B_r^{+}(t,E)$ of radius $r$ centered at $E$ as the family of triangulations with holes obtained by keeping all faces having a vertex at graph distance at most $r-1$ from the starting point of an edge in $E$.
\end{definition1}

For $(t,e),(t',e') \in \mathcal{T}^1$, we define the strong local distance by
\begin{align}\label{strong_local_topology}
d^{+}_{\mathrm{loc}}((t,e),(t',e'))
= \frac{1}{1+\sup\{r \ge 0 : B_r^{+}(t,e) = B_r^{+}(t',e')\}}.
\end{align}
We denote by $\overline{\mathcal{T}^{1}}^{+}$ the completion of $\mathcal{T}^1$ for $d^{+}_{\mathrm{loc}}$. The drawback of this definition is that Theorem~\ref{local_limit_boundary} does not hold for this distance, since a sequence of triangulations rooted on a boundary whose perimeter tends to infinity cannot converge in this distance.
Indeed, the strong ball contains the entire boundary on which the root edge lies. Therefore, we need another notion of local distance for which Theorem~\ref{local_limit_boundary} holds.

\begin{definition1}
Fix $(t,e) \in \mathcal{T}^1$, and let $\rho$ denote the starting point of $e$. We define the \emph{ball of radius $r$} of $t$ centered at $e$ as the map $B_r(t,e)$ consisting of all edges having at least one endpoint at graph distance at most $r-1$ from $\rho$.
\end{definition1}

We then introduce the \emph{local distance} on $\mathcal{T}^1$ by
\begin{align}\label{distance_locale}
d_{\mathrm{loc}}((t,e),(t',e'))
= \big(1 + \max \{r \ge 0 : B_r(t,e) = B_r(t',e')\}\big)^{-1}.
\end{align}
Its completion, denoted by $\overline{\mathcal{T}^1}$, is a Polish space. However, this space is not compact. The drawback of this definition is that the ball of radius $r$ is not a triangulation with holes in general (see Figure~\ref{ball_local}). Note that $B_r^{+}(t,e)$ is obtained from $B_r(t,e)$ by gluing all the faces incident to an edge in $B_r(t,e)$.

\begin{figure}[H]
\centering
\includegraphics[scale=0.2]{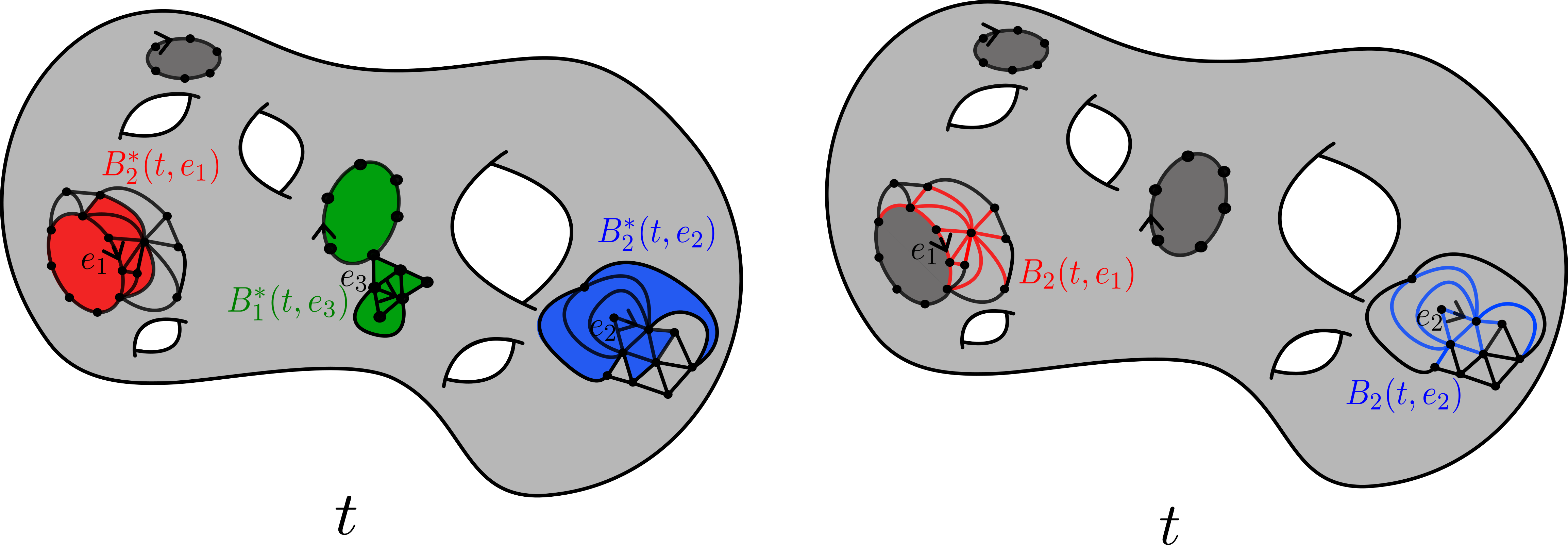}
\caption{We represent a triangulation $t$ of the $\mathbf{p}$-gon. On the left, we represent the strong balls of radius $2$ centered at $e_1$ (in red), $e_2$ (in blue), and $e_3$ (in green). Note that $B_2^{+}(t,e_1)$ contains the boundary on which $e_1$ lies. Moreover, in $B_1^{+}(t,e_3)$, one of the green internal faces shares a vertex with a boundary face, and thus, this boundary face also belongs to $B_2^{+}(t,e_3)$. On the right, we represent in red the ball of radius $2$ centered at $e_1$, which lies on a boundary, and in blue at $e_2$, which lies far from the boundaries. }
\label{ball_local}
\end{figure}

\subsection{Triangulations of the half-plane}

A \emph{triangulation of the half-plane} is a planar triangulation of the $\infty$-gon with finite vertex degrees that is one-ended (see the left part of Figure~\ref{halfplane}). Fix a triangulation with holes $t$ that consists of one infinite boundary, one infinite hole, and a finite number of triangles. We use the shorthand notation
\[
|\partial^{*}t| - |\partial t|
:= |\partial^{*}t \setminus \partial t| - |\partial t \setminus \partial^{*}t|.
\]
In Figure~\ref{halfplane}, it corresponds to the length of the red segment minus the length of the blue segment.
We also denote by $|t|_{\mathrm{in}}$ the number of vertices of $t$ that do not lie on $\partial t$. In Figure~\ref{halfplane}, it corresponds to the number of green vertices.
\begin{figure}[H]
\centering
\includegraphics[scale=0.18]{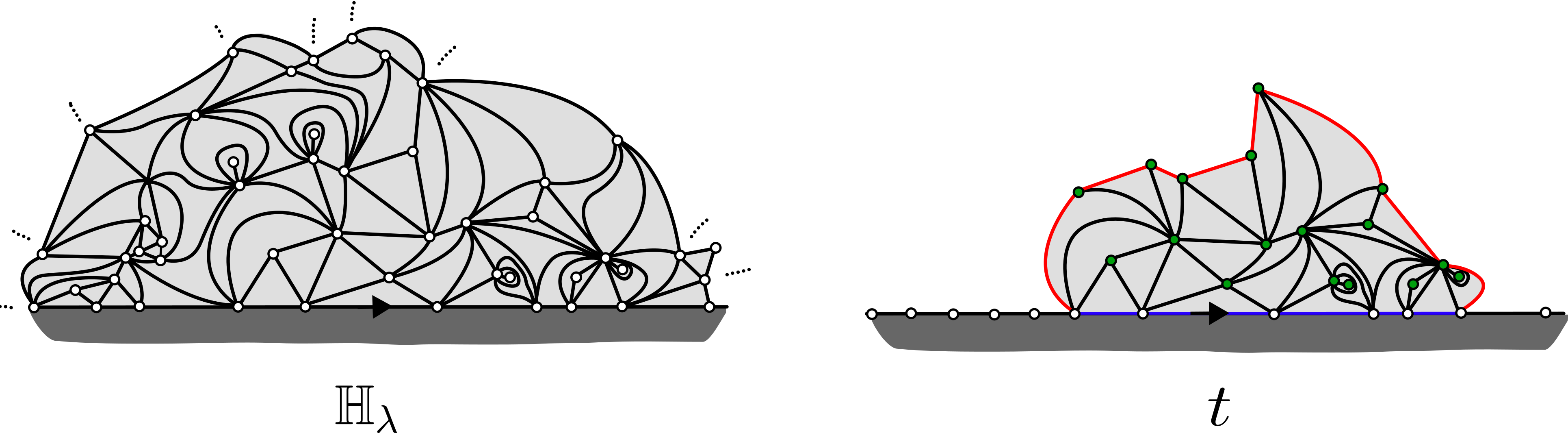}
\caption{On the left: a realisation of $\mathbb{H}_{\lambda}$. On the right, a triangulation $t$ with one infinite boundary (dark grey face) and one infinite hole (white face) such that $t \subset \mathbb{H}_{\lambda}$. On this example, $|t|_{\mathrm{in}} = 17$ denotes the number of green vertices and $|\partial^{*}t|-|\partial t| = 8 - 5 = 3$ denotes the number of red edges minus the number of blue edges.  }
\label{halfplane}
\end{figure}
For any $0 < \lambda \le \lambda_c$, let $h(\lambda)$ be the unique $h \in (0,\tfrac{1}{4}]$ such that
$\lambda = \frac{h}{(1+8h)^{3/2}}$, and define
\begin{align*}
d(\lambda)
= \frac{h \log\!\big(\frac{1+\sqrt{1-4h}}{1-\sqrt{1-4h}}\big)}{(1+8h)\sqrt{1-4h}}.
\end{align*}

For $0 < \lambda \le \lambda_c$, we introduce the type-I random triangulation of the half-plane, denoted by $\mathbb{H}_{\lambda}$, introduced in \cite{budzinski_geodesics} as the type-I analogue of \cite{Angel_Ray}. These triangulations are characterized by
\begin{align}\label{characterize_halfplane}
\mathbb{P}(t \subset \mathbb{H}_{\lambda})
= \bigg(8+\frac{1}{h}\bigg)^{|\partial^{*}t|-|\partial t|}
\lambda^{|t|_{\mathrm{in}}},
\end{align}
where $\lambda = \frac{h}{(1+8h)^{3/2}}$. In particular, the half-plane triangulation $\mathbb{H}_{\lambda}$ satisfies a spatial Markov property: for any $t$, conditionally on $t \subset \mathbb{H}_{\lambda}$, we have
$\mathbb{H}_{\lambda} \setminus t \overset{(d)}{=} \mathbb{H}_{\lambda}$.
We now state the local limit result obtained in \cite[Theorem 1.1]{lions2026locallimitsuniformtriangulations}.
    \begin{theorem}\label{local_limit_boundary}
    Fix $\theta \in [0,\frac{1}{2})$, $\displaystyle \frac{g_n}{n} \underset{n \to +\infty}{\longrightarrow}\theta$ and $\mathbf{p}^{n} = (p^n_{1},\cdots,p^n_{\ell_n})$ such that $|\mathbf{p}^{n}| = o(n)$ and $\displaystyle \frac{|\mathbf{p}^{n}|}{\ell_n}  \underset{n \to +\infty}{\longrightarrow}+\infty$. Denote by $\en$ a uniformly chosen oriented edge on the union $\partial_1\cup \cdots \cup \partial_{\ell_n}$ of the boundaries of $\Tngp$. The following convergence holds for the local topology
\begin{align*}
   (\Tngp,\en) \overset{(d)}{\to}\mathbb{H}_{\lambda(\theta)},
\end{align*}
where $\lambda(\theta)$ is the unique solution to the equation $d(\lambda(\theta)) = \frac{1-2\theta}{6}$.
    \end{theorem}

\subsection{Peeling exploration}\label{peeling_exploration}

In this subsection, we define the peeling exploration for triangulations of the half-plane. Although the peeling exploration can be defined in a more general setting, we restrict ourselves here to the half-plane case.

A peeling algorithm is a function $\mathcal{A}$ that associates to any triangulation with holes $t$ having one infinite boundary, one infinite hole, and finitely many faces, an edge $\mathcal{A}(t)$ on the hole of $t$. Given a peeling algorithm $\mathcal{A}$ and a half-plane triangulation $h$, we define the peeling exploration $(\overline{\mathcal{E}}_k^{\mathcal{A}}(h))_{k \ge 0}$ as follows:
\begin{itemize}
\item[$\bullet$] We define $\overline{\mathcal{E}}_0^{\mathcal{A}}(h)$ as the rooted triangulation with holes consisting of an infinite boundary and an infinite hole.
\item[$\bullet$] For $k \ge 0$, assume that $\overline{\mathcal{E}}_k^{\mathcal{A}}(h)$ has been constructed and satisfies $\overline{\mathcal{E}}_k^{\mathcal{A}}(h) \subset h$. We then define $\overline{\mathcal{E}}_{k+1}^{\mathcal{A}}(h)$ as the triangulation obtained from $\overline{\mathcal{E}}_k^{\mathcal{A}}(h)$ by discovering the triangle incident to the edge $\mathcal{A}(\overline{\mathcal{E}}_k^{\mathcal{A}}(h))$ in $h$ and filling the potential finite hole created.
\end{itemize}

Fix $\lambda \in (0,\lambda_c]$ and a peeling algorithm $\mathcal{A}$. For any $k \ge 1$, we say that
\begin{enumerate}
\item \label{eventC}
the event $\mathbf{C}_k^{\mathcal{A}}$ occurs when the $k$-th discovered triangle has its third vertex not lying on the hole of $\overline{\mathcal{E}}_{k-1}^{\mathcal{A}}(\mathbb{H}_{\lambda})$,
\item \label{eventL}
for any $i \ge 0$, the event $\mathbf{L}_{k,i}^{\mathcal{A}}$ (resp. $\mathbf{R}_{k,i}^{\mathcal{A}}$) occurs when the $k$-th discovered triangle has its third vertex lying on the hole of $\overline{\mathcal{E}}_{k-1}^{\mathcal{A}}(\mathbb{H}_{\lambda})$ exactly $i$ edges to the left (resp. to the right) of $\mathcal{A}(\overline{\mathcal{E}}_{k-1}^{\mathcal{A}}(\mathbb{H}_{\lambda}))$.
\end{enumerate}

Moreover, conditionally on $\overline{\mathcal{E}}_{k-1}^{\mathcal{A}}(\mathbb{H}_{\lambda})$, these events occur with respective probabilities
\begin{align}\label{peeling_transitions}
&\mathbb{P}(\mathbf{C}_k^{\mathcal{A}} \mid \overline{\mathcal{E}}_{k-1}^{\mathcal{A}}(\mathbb{H}_{\lambda}))
= \lambda\bigg(8+\frac{1}{h}\bigg), \nonumber\\
&\mathbb{P}(\mathbf{L}_{k,i}^{\mathcal{A}} \mid \overline{\mathcal{E}}_{k-1}^{\mathcal{A}}(\mathbb{H}_{\lambda})) = \mathbb{P}(\mathbf{R}_{k,i}^{\mathcal{A}} \mid \overline{\mathcal{E}}_{k-1}^{\mathcal{A}}(\mathbb{H}_{\lambda}))
= \bigg(8+\frac{1}{h}\bigg)^{-i} w_{\lambda}(i+1).
\end{align}

\subsection{Hulls}\label{section_hulls}

Fix $n,g \ge 0$ such that $n \ge 2g-1$. Also fix $\ell > 0$ and $(p_1,\ldots,p_{\ell}) \in (\mathbb{N}^{*})^{\ell}$. For $t \in \mathcal{T}_{\mathbf{p}}(n,g)$ and a subset $E$ of oriented edges of $t$, recall the definition of $B_r^{+}(t,E)$ from Definition~\ref{strong_ball}. Then $t \setminus B_r^{+}(t,E)$ is a family $t^1,\ldots,t^k$ of triangulations of multipolygons. Let us recall that for $t_0$ a family of triangulations with holes, we denote by $|t_0|$ the total number of internal triangles composing the family.

Recall that $t$ has $2n-\sum_{i=1}^{\ell}(p_i-2)$ triangles. It follows that at most one connected component $t^i$ can have more than $n$ triangles. Therefore, gluing to $B_r^{+}(t,E)$ all the connected components $t^j$ with at most $n$ triangles defines a family $B_r^{\bullet}(t,E)$ of triangulations with holes such that $t \setminus B_r^{\bullet}(t,E)$ is either empty or connected.

\begin{definition1}
The subset $B_r^{\bullet}(t,E)$ is called the hull of radius $r$ centered at $E$.
\end{definition1}

\begin{figure}[H]
\includegraphics[scale=0.2]{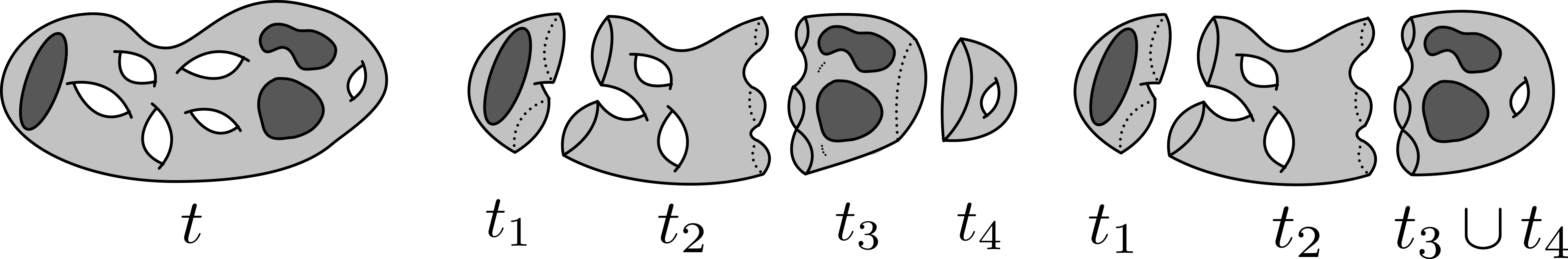}
\caption{On the left, a triangulation $t$ of the $\mathbf{p}$-gon. In the middle, we have
$B_r^{+}(t,\partial t)=(t_1,t_3)$ and
$t \setminus B_r^{+}(t,\partial t)=(t_2,t_4)$.
On the right, we represent the hull of radius $r$, for which
$B_r^{\bullet}(t,\partial t)=(t_1,t_3 \cup t_4)$ and
$t \setminus B_r^{\bullet}(t,\partial t)=t_2$.}
\label{hulls}
\end{figure}

In the remainder of the paper, the key task is to understand the volume growth of strong balls, that is, the quantity $|B_r^{+}(\Tng,e_n)|$, where $e_n$ is an oriented edge chosen uniformly at random on $\Tng$. However, this is a difficult problem, since $\Tng \setminus B_r^{+}(\Tng,e_n)$ may consist of many connected components $T^1,\ldots,T^k$, whose behaviour is hard to control simultaneously. Thus, we will rather consider the growth of the hulls $|B_r^{\bullet}(\Tng,e_n)|$.

For the rest of the paper, we fix $\displaystyle \frac{g_n}{n} \to \theta \in (0,\tfrac{1}{2})$. We write
$B_{r}^{+}(n) := B_r^{+}(\Tng,e_n)$ and
$B_{r}^{\bullet}(n) := B_r^{\bullet}(\Tng,e_n)$,
where $e_n$ denotes the root edge of $\Tng$. Most of the time, we simply write $B_r^{+}:= B_{r}^{+}(n)  $ and $B_r^{\bullet}:=B_{r}^{\bullet}(n) $, since the dependence on $n$ is clear from the context. It is direct to verify that
\[
B_{r+1}^{\bullet}
= B_r^{\bullet} \cup B_1^{\bullet}(T',\partial T')
\quad\text{and}\quad
\partial^{*}B_{r+1}^{\bullet}
= \partial^{*}B_1^{\bullet}(T',\partial T'), \hspace*{0.3cm} \text{ where }T' = \Tng \setminus B_r^{\bullet}.
\]
Thus, understanding $B_1^{\bullet}(T',\partial T')$ will play a crucial role in the study of the growth of $|B_r^{\bullet}|$, which is the focus of Section~\ref{section_boundary}.

\section{From $(1-\varepsilon)D_{\theta}\log(n)$ to $D_{\theta}\log(n)$}\label{section_final_proof}
In this section, we give the proof of Theorem~\ref{theorem_typical_distances} given the following theorem that we prove in Section~\ref{section_growth_of_hulls}.

\begin{theorem}\label{growth_of_balls}
For any $\varepsilon > 0$, we have the convergence
\begin{align*}
\frac{\log |B_{(1-\varepsilon)D_{\theta}\log(n)}^{\bullet}(\Tng,e_n)|}{(1-\varepsilon)\log(n)}
\underset{n \to +\infty}{\overset{(\mathbb{P})}{\longrightarrow}} 1,
\end{align*}
where $D_{\theta} = \frac{1}{\log(m_{\theta}^{-1})}$.
\end{theorem}

We also need a proposition that shows that the volumes of hulls are indeed a good approximation for those of balls. 
\begin{proposition}\label{hull_well_defined}
For any $\varepsilon > 0$ and $n \ge 1$, let us assume that the event $\mathbf{Iso}_{n}(\frac{\varepsilon}{2})$ occurs. Let us also fix $r \ge 0$ such that $|B_r^{+}| \le n^{1-\varepsilon}$. Then, we have $B_r^{\bullet} \neq \Tng$. Moreover, if $K_{\theta}\log(n) \le |B_r^{\bullet}|$, then we have
\begin{align}\label{approx_balls_with_hulls}
\big|\log(|B_r^{\bullet}|)-\log(|\partial^{*}B_r^{\bullet}|)\big| \le \log(3\ct^{-1}).
\end{align}
In particular, we have 
\[
\big|\log(|B_r^{+}|)-\log(|B_r^{\bullet}|)\big| \le \log(3\ct^{-1}).
\]
\end{proposition}
\begin{proof}
For $r \ge 0$, let $T^1,\dots,T^k$ denote the connected components of $\Tng \setminus B_r^{+}$ ordered by increasing number of triangles (we choose an arbitrary convention to order the $T^i$ in case of equality for the number of triangles). Denote by $B_r^{\circ}$ the triangulation obtained by gluing to $B_r^{+}$ the components $T^1,\dots,T^{k-1}$. 
For $1 \le i \le k-1$, applying the isoperimetric inequality to the component $T^i$, we obtain $$|T^i| \le \max(K_{\theta}\log(n),\ct^{-1}|\partial T^i|).$$ Moreover, using the crude bound $ |\partial^{*}B_r^{+}| \le 3 |B_r^{+}| \le 3n^{1-\varepsilon}$, for $n$ large enough, we obtain
\begin{align*}
|B_r^{\circ}| &= |B_r^{+}|+\somme{i=1}{k-1}{|T^i|} \\&\le n^{1-\varepsilon} + 3n^{1-\varepsilon} \cdot K_{\theta} \log(n) + \ct^{-1}|\partial^{*}B_r^{+}|\\
&\le (1+3K_{\theta}\log(n) +3\ct^{-1})n^{1-\varepsilon} \\&\le n^{1-\frac{\varepsilon}{2}}.
\end{align*}
It follows that $|T^k| \ge n$. Hence,
\[
B_r^{\bullet} = B_r^{\circ} \neq \Tng.
\]

Moreover, if $K_{\theta}\log(n) \le |B_r^{\bullet}|$, the isoperimetric inequality yields $ \ct|B_r^{\bullet}|\le |\partial^{*}B_r^{\bullet}|$ and we have the crude bound $ |\partial^{*}B_r^{\bullet}| \le 3|B_r^{\bullet}|$, which gives \eqref{approx_balls_with_hulls}. The second part of the proposition follows from the inclusion $\partial^{*} B_r^{\bullet} \subset B_r^{+} \subset B_r^{\bullet}$.
\end{proof}

Now, we can prove Theorem~\ref{theorem_typical_distances}.
\begin{proof}
Let us fix $\varepsilon > 0$. We start by proving the theorem in the case where $x_n^1,x_n^2$ denote the starting points of two oriented edges $e_n^1,e_n^2$ chosen uniformly at random on $\Tng$. In this proof, we write $B_r := B_r(\Tng,e_n^1)$ and $B_r^{\bullet} =  B_r^{\bullet}(\Tng,e_n^1)$. We start by bounding $\mathbb{P}(d_{\Tng}(x_n^1,x_n^2) \le (1-\varepsilon)D_{\theta}\log(n))$. We can write 
\begin{align}\label{lower_bound}
\mathbb{P}(d_{\Tng}(x_n^1,x_n^2) \le (1-\varepsilon)D_{\theta}\log(n)) &= \mathbb{P}( x_n^2\in B_{(1-\varepsilon)D_{\theta}\log(n)}) \nonumber\\&\le \mathbb{P}( e_n^2 \in B_{(1-\varepsilon/2)D_{\theta}\log(n)}^{\bullet}) \nonumber\\
&= \mathbb{E}\bigg[\frac{2\#E(B_{(1-\varepsilon/2)D_{\theta}\log(n)}^{\bullet})}{6n}\bigg] \nonumber\\
& \le  \mathbb{E}\bigg[\frac{4|B_{(1-\varepsilon/2)D_{\theta}\log(n)+1}^{\bullet}|}{6n}\bigg],
\end{align}
since each edge is incident to at most $2$ triangles. It follows that the right-hand side is bounded by $
\displaystyle \mathbb{E}\bigg[\frac{4|B_{(1-\varepsilon/2)D_{\theta}\log(n)+1}^{\bullet}|}{6n}\bigg]$. Using Theorem~\ref{growth_of_balls} and the fact that $|B_{(1-\varepsilon/2)D_{\theta}\log(n)+1}^{\bullet}| \le 6n$, we deduce that \eqref{lower_bound} goes to $0$ as $n$ tends to $+\infty$.

Let us now prove that for any $\eta > 0$, there exists a constant $C_{\theta,\eta} > 0$ which does not depend on $\varepsilon$ such that for $n$ large enough, we have 

\begin{align}\label{toshow_eta}
\mathbb{P}(d_{\Tng}(x_n^1,x_n^2) \ge (D_{\theta}+C_{\theta,\eta}\varepsilon)\log(n)) \le 4\eta.
\end{align}
It is easy to check that if this holds for any $\varepsilon,\eta > 0$, letting first $\varepsilon \to 0$ and then $\eta \to 0$ concludes the proof of the first part of Theorem~\ref{theorem_typical_distances}.

Fix $\eta > 0$. For $\beta > 0$ we recall that the event $\mathbf{Iso}_n(\frac{\varepsilon}{4},\eta,\beta)$ occurs when the conclusions of Theorem~\ref{isoperimetric_inequalities} and Proposition~\ref{multicurves_few_boundaries} are satisfied and we have $\mathrm{Isol}_{\beta}(\Tng) \le \eta n$. Using Corollary~\ref{Iso_typical}, we fix $\beta > 0$ such that for $n$ large enough we have
\begin{align*}
\mathbb{P}\bigg(\mathbf{Iso}_n\bigg(\frac{\varepsilon}{4},\eta,\beta\bigg)\bigg)\ge 1-\eta.
\end{align*}
Let us now prove that there exists a constant $C_{\theta,\eta} > 0$ such that, for $n$ large enough, under $\mathbf{Iso}_n(\frac{\varepsilon}{4},\eta,\beta) $ and  $\{n^{1-2\varepsilon}\le|B_{(1-\varepsilon)D_{\theta}\log(n)}^{\bullet}|\le n^{1-\varepsilon/2}\}$, we have
 \begin{center}
 $|\Tng \setminus  B_{(D_{\theta}+C_{\theta,\eta}\varepsilon)\log(n)}^{+}|< 2\eta n.$
 \end{center} 
 First, by Proposition~\ref{hull_well_defined}, we have $|B_{(1-\varepsilon)D_{\theta}\log(n)}^{+}| \ge \frac{\ct}{3} n^{1-2\varepsilon}$. Now, for $r \ge (1-\varepsilon)D_{\theta}\log(n)$ such that $|\Tng \setminus  B_{r}^{+}|\ge 2\eta n$, we write $T^1,\cdots,T^k$ for the connected components of $\Tng \setminus  B_{r}^{+}$. Then there are two cases :
\begin{itemize}
	\item[$\bullet$] If one of the components $T^i$ has at least $n$ triangles, then using $\mathbf{Iso}_n(\frac{\varepsilon}{4},\eta,\beta)$, we find $|\partial^{*}B_r^{+}| \ge \ct |B_r^{+}|$. It follows that $|B_{r+1}^{+}| \ge (1+\frac{\ct}{3})|B_r^{+}|$.
	\item[$\bullet$] If all the components $T^i$ have at most $n$ triangles, let us denote by $\ell_i$ the total length of the multicurve that separates $T^i$ from $B_r^{+}$. Then if $\ell_i \le \beta |T^i|$, using the fact that $|B_r^{+}| \ge \sqrt{n}$, we deduce that all the triangles in $T^i$ are $\beta$-isolated.  Using $\mathbf{Iso}_n(\frac{\varepsilon}{4},\eta,\beta)$, this implies $\displaystyle \somme{\substack{i=1\\\ell_i \le \beta |T^i|}}{k}{|T^i|} \le \eta n$. Using this with the fact that $\somme{i=1}{k}{|T^i|} \ge 2\eta n$, we obtain $\displaystyle \somme{\substack{i=1\\\ell_i \ge \beta |T^i|}}{k}{|T^i|} \ge \eta n$ and thus $\somme{i=1}{k}{\ell_i} \ge \beta \eta n$. This rewrites as $|\partial^{*}B_r^{+}| \ge \beta \eta n$. It follows that $|B_{r+1}^{+}| \ge |B_r^{+}|+\frac{1}{3}\beta \eta  n \ge (1+\frac{\beta \eta}{6}) |B_r^{+}|$.
\end{itemize}
It follows that for $r' \ge 0$ such that $|\Tng \setminus  B_{(1-\varepsilon)D_{\theta}\log(n) + r'}^{+}|\ge 2\eta n$, we have $$|B_{(1-\varepsilon)D_{\theta}\log(n) + r'}^{+}| \ge \frac{\ct}{3}n^{1-2\varepsilon}(1+c_{\theta,\eta})^{r'}, $$ where $c_{\theta,\eta} = \min(\frac{\beta \eta}{6},\frac{\ct}{3})$. Let us introduce $C_{\theta,\eta} = \frac{4}{\log(1+c_{\theta,\eta})}$. We directly deduce that for $r \ge (D_{\theta}+C_{\theta,\eta}\varepsilon )\log(n)$ we have $|\Tng \setminus  B_{r}^{+}| < 2 \eta  n$. Since for $n$ large enough we have 
\begin{center}
$ \mathbb{P}(\mathbf{Iso}_n(\frac{\varepsilon}{4},\eta,\beta) \cap \{n^{1-2\varepsilon}\le|B_{(1-\varepsilon)D_{\theta}\log(n)}^{\bullet}|\le n^{1-\varepsilon/2}\}) \ge 1-3\eta$,
\end{center} 
we deduce 
\begin{align*}
\mathbb{P}(d_{\Tng}(x_n^1,x_n^2) \ge (D_{\theta}+C_{\theta,\eta}\varepsilon)\log(n))  \le \mathbb{E}\bigg[\frac{2|\Tng \setminus B_{(D_{\theta}+C_{\theta,\eta}\varepsilon)\log(n)}^{+}|}{6n} \bigg] \le 4 \eta .
\end{align*}
This proves \eqref{toshow_eta}.

Now, we prove the second part of the theorem, i.e. $y_n^1,y_n^2$ are independent uniform vertices on $\Tng$.  Conditionally on $\Tng$, for $v \in \Tng$ a vertex, we have 
\begin{align*}
&\mathbb{P}(x_n^1 = v\text{ | }\Tng) = \frac{\deg_{\Tng}(v)}{6n} \text{ and }\mathbb{P}(y_n^1 = v\text{ | }\Tng) = \frac{1}{n+2-2g_n}.
\end{align*}
 Thus, for $v,v'$ two vertices in $\Tng$, using the fact that $\frac{g_n}{n} \to \theta$ and $\deg_{\Tng}(v)\ge 1$, we can write for $n$ large enough 
\begin{align*}
\mathbb{P}(y_n^1 = v,y_n^2 = v'\text{ | }\Tng) &= \frac{(6n)^2}{(n+2-2g_n)^2} \cdot \frac{1}{\deg_{\Tng}(v)}\cdot \frac{1}{\deg_{\Tng}(v')}\mathbb{P}(x_n^1 = v,x_n^2 = v'\text{ | }\Tng)\\
&\le 2\frac{36}{(1-2\theta)^2}\mathbb{P}(x_n^1 = v,x_n^2=v'\text{ | }\Tng).
\end{align*} 
In other words, the law of $(y_n^1,y_n^2)$ is absolutely continuous with respect to the law of $(x_n^1,x_n^2)$; the conclusion immediately follows.
\end{proof}

\section{Growth rate at the boundary}\label{section_boundary}

In this section, we define $\Bn := B_1^{+}(\Tngp,\partial \Tngp)$ as the triangulation with holes obtained by gluing all the triangles having at least one vertex on $\partial \Tngp$, as described in Section~\ref{section_hulls}.
We use the shorthand notation $\Bnb:= B_1^{\bullet}(\Tngp,\partial \Tngp)$.
The main results of the section are Proposition~\ref{second_moment_bounded_degree} and Proposition~\ref{growth_rate_first_moment}.
\subsection{Second moment bound}
This section is dedicated to proving Proposition~\ref{second_moment_bounded_degree}. For any $x \in  \Tngp$, we write
\[
D_x^n := \deg_{\Tngp}(x),
\]
where $\deg_{\Tngp}(x)$ denotes the number of oriented edges starting from $x$.
Each edge with one endpoint on $\partial \Tngp$ is incident to at most two triangles, yielding the bounds
\begin{align}\label{bound}
|\partial^{*}\Bnb| \le |\partial^{*}\Bn| \le |\Bn| \le 2\sum_{x \in \partial \Tngp} D_x^n.
\end{align}
We hope the following proposition may be of interest in other contexts; we therefore give a general formulation.

\begin{proposition}\label{second_moment_bounded_degree}
Let $n, g \ge 1$ such that $n \ge 2g-1$ and let $\mathbf{p} = (p_1,\cdots,p_{\ell}) \in (\mathbb{N}^{*})^{\ell}$. Denoting by $e_1$ the root edge on $\partial_1 T_{n,g,\mathbf{p}}$ and by $\rho_1$ its starting point, we have 
\begin{align*}
\mathbb{E}[\deg_{T_{n,g,\mathbf{p}}}(\rho_1)^2] \le 4 + 5 \frac{\tau_{\mathbf{p}}(n+1,g)}{\tau_{\mathbf{p}}(n,g)}+ 16\frac{\tau_{\mathbf{p}}(n+2,g)}{\tau_{\mathbf{p}}(n,g)}.
\end{align*}

Moreover, for any $\theta \in [0,\frac{1}{2})$, there exists $C_{\theta} > 0$ such that  for any $\frac{g_n}{n} \to \theta \in [0,\frac{1}{2})$ and $\mathbf{p}^n = (p_1^n,\cdots,p_{\ell_n}^n) \in (\mathbb{N}^{*})^{\ell_n}$ such that $|\mathbf{p}^n| = o(n)$, for $n$ large enough
\begin{align*}
\forall x \in \partial \Tngp, \qquad \mathbb{E}[(D_x^n)^2] \le C_{\theta}.
\end{align*}
Therefore by \eqref{bound}, for $n$ large enough,
\begin{align*}
\mathbb{E}\big[|\partial^{*}\Bnb|^2\big]
\le \mathbb{E}\big[|\Bn|^2\big]
\le 4C_{\theta}|\mathbf{p}^{n}|^2.
\end{align*}
\end{proposition}

\begin{proof}
We start by proving the second inequality assuming the first. Let $\frac{g_n}{n} \to \theta \in [0,\frac{1}{2})$ and $\mathbf{p}^n = (p_1^n,\cdots,p_{\ell_n}^n) \in (\mathbb{N}^{*})^{\ell_n}$ such that $|\mathbf{p}^n| = o(n)$. Using the translation invariance of $(\Tngp,\en_i)$ along the boundary $\partial_i \Tngp$, for any $1 \le i \le \ell_n$, we have 
\begin{align*}
\forall x \in \partial \Tngp, \qquad \mathbb{E}[(D_x^n)^2] \le 4 + 5 \frac{\tau_{\mathbf{p}^n}(n+1,g_n)}{\tau_{\mathbf{p}^n}(n,g_n)}+ 16\frac{\tau_{\mathbf{p}^n}(n+2,g_n)}{\tau_{\mathbf{p}^n}(n,g_n)}.
\end{align*}
We conclude by bounding the right-hand side using Lemma~\ref{bounded_ratio_vertices}. The last inequality follows from \eqref{bound} and Cauchy-Schwarz.

Now, let us prove the first inequality. Let $n, g \ge 1$ such that $n \ge 2g-1$ and let $\mathbf{p} = (p_1,\cdots,p_{\ell}) \in (\mathbb{N}^{*})^{\ell}$. Fix $t \in \mathcal{T}_{\mathbf{p}}(n,g)$. Let $\rho_1,\ldots,\rho_{\ell}$ denote the starting points of the root edges $e_1,\ldots,e_{\ell}$ on the boundaries of $t$. Fix an oriented edge $e$ such that $e \notin \partial t$ and the starting point of $e$ is $\rho_1$. Let $y$ denote the endpoint of $e$.

We denote by $t \setminus e$ the family of triangulations obtained by cutting along $e$. We define a mapping $\varphi(t,e)$ as follows.

\medskip

\noindent
\begin{minipage}[c]{0.4\textwidth}
\textbf{1.} If $y \notin \partial t$, then $t \backslash e \in \mathcal{T}_{(p_1+2,p_2,\ldots,p_{\ell})}(n+1,g)$, where the first boundary is still rooted at $e_1$. We define $\varphi(t,e)=t \backslash e$.
\end{minipage}
\hfill
\begin{minipage}[c]{0.5\textwidth}
\centering
\includegraphics[width=\linewidth]{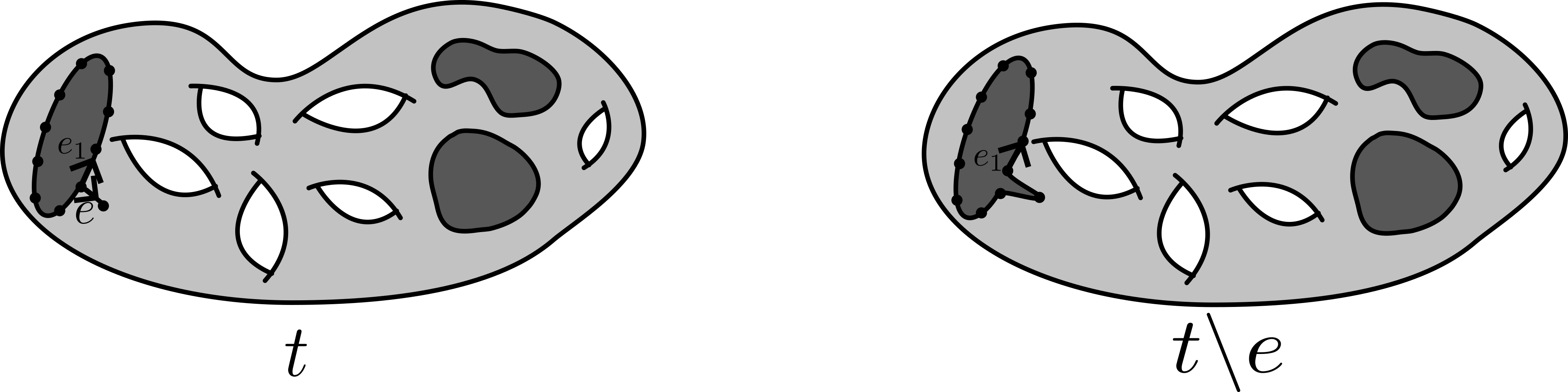}
\end{minipage}

\medskip

\noindent
\begin{minipage}[c]{0.4\textwidth}
\textbf{2.} If $y \in \partial_i t$ with $i \ge 2$, then cutting along $e$ defines a map
\[
t \backslash e \in \mathcal{T}_{(p_1+p_i+2,\ldots,p_{i-1},p_{i+1},\ldots,p_{\ell})}(n+2,g).
\]
The first boundary of $t \backslash e$ is rooted at $e_1$. Let $k \in \{0,\ldots,p_i-1\}$ be such that $y$ is the $k$-th vertex to the right of $\rho_i$ on $\partial_i t$. We define $\varphi(t,e)=(t \backslash e,i,k)$.
\end{minipage}
\hfill
\begin{minipage}[c]{0.5\textwidth}
\centering
\includegraphics[width=\linewidth]{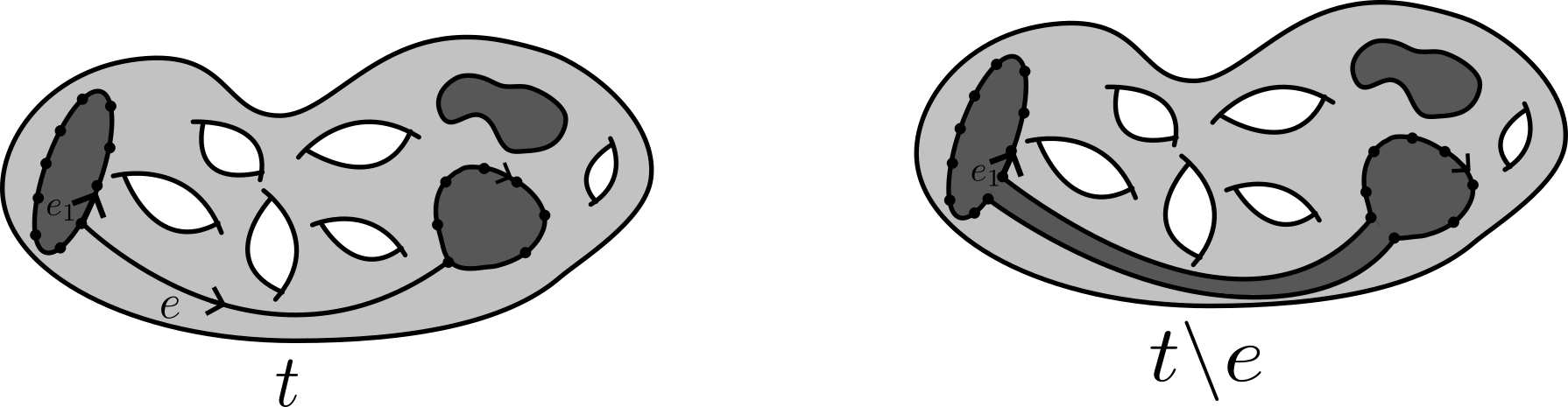}
\end{minipage}

\medskip

\noindent\textbf{3.} If $y \in \partial_1 t$, let $k \in \{0,\ldots,p_1-1\}$ be such that $y$ is the $k$-th vertex to the right of $\rho_1$ on $\partial_1 t$. There are two cases.

\medskip

\noindent
\begin{minipage}[c]{0.4\textwidth}
\textbf{3.1.} If internal faces of $t \backslash e$ form a connected subset of $(t \backslash e)^{*}$, then
\[
t \backslash e \in \mathcal{T}_{p_1-k+1,k+1,p_2,\ldots,p_{\ell}}(n,g-1).
\]
We define $\varphi(t,e) = t \backslash e$.
\end{minipage}
\hfill
\begin{minipage}[c]{0.5\textwidth}
\centering
\includegraphics[width=\linewidth]{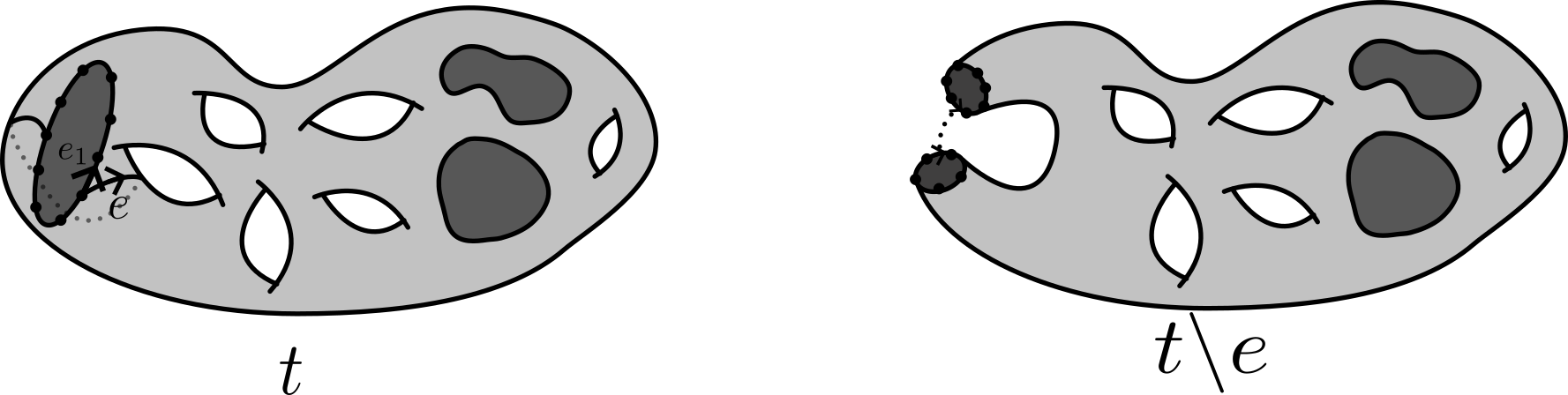}
\end{minipage}

\medskip

\noindent
\begin{minipage}[c]{0.4\textwidth}
\textbf{3.2.} If internal faces of $t \backslash e$ form two connected subset of $(t \backslash e)^{*}$, then $
t \backslash e = (t_1,t_2) \in
\mathcal{T}_{p_1-k+1,\mathbf{p}_I}(n_1,g_1)
\times
\mathcal{T}_{k+1,\mathbf{p}_J}(n_2,g_2),
$
where $n_1+n_2=n$, $g_1+g_2=g$, and $I \sqcup J=\{2,\ldots,\ell\}$. We define $\varphi(t,e)=(t_1,t_2)$.
\end{minipage}
\hfill
\begin{minipage}[c]{0.5\textwidth}
\centering
\includegraphics[width=\linewidth]{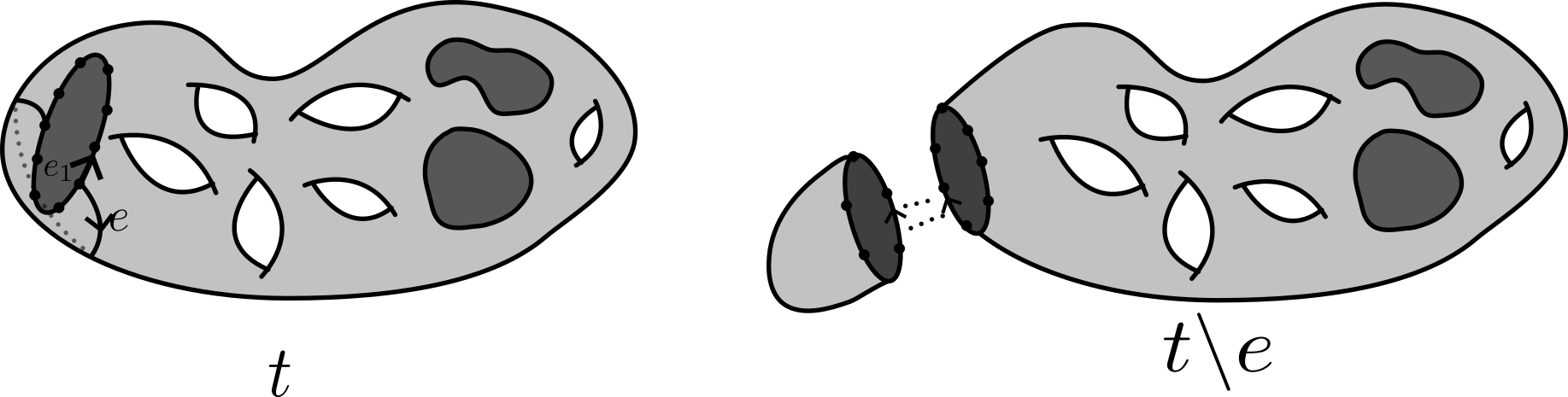}
\end{minipage}

\medskip

We claim that $\varphi$ defines an injective mapping. Indeed, given $\varphi(t,e)$, one can recover $e$ and then $t$ by gluing the appropriate edges. This yields
\begin{align}\label{bound_with_ABCD}
&\#\Big\{(t,e): t \in \mathcal{T}_{\mathbf{p}}(n,g),\ e \text{ oriented edge with starting point } \rho_1,\ e \notin \partial t\Big\} \\
&\le
\underbrace{\tau_{(p_1+2,p_2,\ldots,p_{\ell})}(n+1,g)}_{A}
+
\underbrace{\sum_{i=2}^{\ell} p_i \,
\tau_{(p_1+p_i+2,\ldots,p_{i-1},p_{i+1},\ldots,p_{\ell})}(n+2,g)}_{B}
\nonumber\\
&\quad+
\underbrace{\sum_{k=0}^{p_1-1}
\tau_{p_1-k+1,k+1,p_2,\ldots,p_{\ell}}(n,g-1)}_{C}
+
\underbrace{\sum_{k=0}^{p_1-1}
\sum_{\substack{n_1+n_2=n\\ g_1+g_2=g\\ I \sqcup J=\{2,\ldots,\ell\}}}
\tau_{p_1-k+1,\mathbf{p}_I}(n_1,g_1)
\tau_{k+1,\mathbf{p}_J}(n_2,g_2)}_{D}.
\nonumber
\end{align}

Now, we claim that we have the bound
\begin{align}\label{bound_ABCD}
\max(A,B,C,D) \le \tau_{\mathbf{p}}(n+1,g).
\end{align}

We only prove that $D \le \tau_{\mathbf{p}}(n+1,g)$. The argument is identical for the other terms. Consider the random triangulation $T_{n+1,g,\mathbf{p}}$ and the triangle incident to the root edge $e_1$. Let $\mathcal{D}$ denote the event that this triangle has its third vertex on $\partial_1$ and disconnects $T_{n+1,g,\mathbf{p}}$. The probability that this occurs is given by
\begin{align*}
1 \ge \mathbb{P}(\mathcal{D}) = \tau_{\mathbf{p}}(n+1,g)^{-1}
\sum_{k=0}^{p_1-1}
\sum_{\substack{n_1+n_2=n\\ g_1+g_2=g\\ I \sqcup J=\{2,\ldots,\ell\}}}
\tau_{p_1-k,\mathbf{p}_I}(n_1,g_1)
\tau_{k+1,\mathbf{p}_J}(n_2,g_2).
\end{align*}
Using the inequality
$\tau_{p_1-k+1,\mathbf{p}_I}(n_1,g_1)
\le
\tau_{p_1-k,\mathbf{p}_I}(n_1,g_1)$
yields \eqref{bound_ABCD}.

\medskip
This is enough to bound $\mathbb{E}[\deg_{T_{n,g,\mathbf{p}}}(\rho_1)]$. For the second moment, let us bound
\begin{align*}
\#\Big\{(t,e,e'): t \in \mathcal{T}_{\mathbf{p}}(n,g),\ e,e' \text{ are oriented edges starting at } \rho_1 \text{ and } e,e' \notin \partial t\Big\}.
\end{align*}
For $(t,e,e')$ with $e \neq e'$,  let us distinguish the four cases as in the first part of the proof for the pair $(t,e)$. In each case, we define the mapping $\psi(t,e,e')$:
\begin{itemize}
\item[$\mathbf{1}.$] $\psi(t,e,e')=\varphi(\varphi(t,e),e')$.
\item[$\mathbf{2}.$] If $\varphi(t,e)=(t \backslash e,i,k)$, then $\psi(t,e,e')=(\varphi(t \backslash e,e'),i,k)$.
\item[$\mathbf{3.1}.$] $\psi(t,e,e')=\varphi(\varphi(t,e),e')$.
\item[$\mathbf{3.2}.$] If $\varphi(t,e)=(t_1,t_2)$ and $e' \in t_1$, then $\psi(t,e,e')=(\varphi(t_1,e'),t_2)$. If $e' \in t_2$, then $\psi(t,e,e')=(\varphi(t_2,e'),t_1)$.
\end{itemize}
If $e=e'$, we simply set $\psi(t,e,e')=\varphi(t,e)$. It is straightforward to check that the function $\psi$ is injective.

Using \eqref{bound_ABCD}, each resulting case contributes at most $\tau_{\mathbf{p}}(n+2,g)$, yielding
\begin{align*}
\#\Big\{(t,e,e'): t \in \mathcal{T}_{\mathbf{p}}(n,g),\ e,e' &\text{ are oriented edges with starting point } \rho_1 \text{ and } e,e' \notin \partial t \text{ and } e \neq e'\Big\}\\&
\le 16 \tau_{\mathbf{p}}(n+2,g).
\end{align*}

Therefore,
\begin{align*}
\mathbb{E}[(D_{\rho_1}^n)^2]
&= 4
+\frac{\#\Big\{(t,e,e'): t \in \mathcal{T}_{\mathbf{p}}(n,g),\ e,e' \text{ are oriented edges with starting point } \rho_1,\ e,e' \notin \partial t,\ e\neq e'\Big\}}{\tau_{\mathbf{p}}(n,g)}\\
&+5\frac{\#\Big\{(t,e): t \in \mathcal{T}_{\mathbf{p}}(n,g),\ e \text{ are oriented edges with starting point } \rho_1,\ e \notin \partial t\Big\}}{\tau_{\mathbf{p}}(n,g)}\\
&\le 4 + 5 \frac{\tau_{\mathbf{p}}(n+1,g)}{\tau_{\mathbf{p}}(n,g)}+ 16\frac{\tau_{\mathbf{p}}(n+2,g)}{\tau_{\mathbf{p}}(n,g)},
\end{align*}
where the term $4$ accounts for pairs of boundary edges incident to $\rho_1$ and the factor $5$ counts the cases where $e = e' \notin \partial t$ or $e \in \partial t$ and $e' \notin \partial t$ or $e' \in \partial t$ and $e \notin \partial t$. This concludes the proof.
\end{proof}
In the next section, we aim to establish a first moment estimate that provides the correct growth rate for the quantity $\displaystyle \frac{|\partial^{*}\Bnb|}{|\mathbf{p}^{n}|}$.

\subsection{First moment estimate}\label{section_first_moment}
This section is dedicated to proving the next proposition. It provides an $L^1$ convergence for $
\frac{|\partial^{*}\Bnb|}{|\mathbf{p}^{n}|}$. Let us recall that for a triangulation with holes $t$, the internal vertices (resp. edges) are the vertices (resp. edges) that do not lie on a hole.
\begin{proposition}\label{growth_rate_first_moment}
For any $\frac{g_n}{n} \to \theta \in [0,\frac{1}{2})$ and $\mathbf{p}^n = (p_1^n,\cdots,p_{\ell_n}^n) \in (\mathbb{N}^{*})^{\ell_n}$ such that $|\mathbf{p}^n| = o(n)$ and $\ell_n = o(|\mathbf{p}^n|)$, we have 
\begin{align*}
\mathbb{E}\bigg[
\indi{\frac{|\partial^{*}\Bnb|}{|\Bnb|}\ge \ct}
\bigg|\frac{|\partial^{*}\Bnb|}{|\mathbf{p}^{n}|} - m_{\theta}^{-1}\bigg|
\bigg]
\underset{n \to +\infty}{\longrightarrow} 0.
\end{align*}
\end{proposition}

We also state an analogue result where the $L^1$ convergence is replaced by convergence in probability. 

\begin{proposition}\label{growth_rate}
Let $\frac{g_n}{n} \to \theta \in [0,\frac{1}{2})$ and $\mathbf{p}^n = (p_1^n,\cdots,p_{\ell_n}^n) \in (\mathbb{N}^{*})^{\ell_n}$ such that $|\mathbf{p}^n| = o(n)$ and $\ell_n = o(|\mathbf{p}^n|)$. For any $\delta > 0$, we have
\begin{align*}
\mathbb{P}\bigg(\frac{|\partial^{*}\Bnb|}{|\Bnb|}\ge \ct,
\ \bigg|\frac{|\partial^{*}\Bnb|}{|\mathbf{p}^n|} - m_{\theta}^{-1}\bigg| \ge \delta
\bigg)
\underset{n \to +\infty}{\longrightarrow} 0.
\end{align*}
\end{proposition}
It is immediate that Proposition~\ref{growth_rate_first_moment} implies Proposition~\ref{growth_rate}. Let us prove Proposition~\ref{growth_rate_first_moment} assuming Proposition~\ref{growth_rate}.
\begin{proof}
By Proposition~\ref{second_moment_bounded_degree}, the random variable $\frac{|\partial^{*}\Bnb|}{|\mathbf{p}^{n}|}$ is bounded in $L^2$. By Proposition~\ref{growth_rate}, the random variable $
\indi{\frac{|\partial^{*}\Bnb|}{|\Bnb|}\ge \ct}
\bigg|\frac{|\partial^{*}\Bnb|}{|\mathbf{p}^{n}|} - m_{\theta}^{-1}\bigg|$ converges in probability to $0$. Putting these two facts together, it also converges in $L^1$ to $0$ which concludes the proof.
\end{proof}

Note that working under the event $\displaystyle \frac{|\partial^{*}\Bnb|}{|\Bnb|}\ge \ct$ will be sufficient for our needs since this will be applied in Section~\ref{section_growth_of_hulls} to $\Tng \setminus  B_r^{\bullet}(\Tng,\en)$ under $\mathbf{Iso}_n(\frac{\varepsilon}{2})$ where $K_{\theta}\log(n) \le |B_r^{\bullet}(\Tng,\en)| \le n^{1-\varepsilon}$. In this context, the condition $\displaystyle \frac{|\partial^{*}\Bnb|}{|\Bnb|}\ge \ct$ rewrites as $\displaystyle \frac{|\partial^{*}B_r^{\bullet}(\Tng,\en)|}{|B_{r+1}^{\bullet}(\Tng,\en)|} \ge \ct,$ which is satisfied under $\mathbf{Iso}_n(\frac{\varepsilon}{2})$ (see Theorem~\ref{isoperimetric_inequalities}).\\

For the rest of the section, we fix $\frac{g_n}{n} \to \theta \in [0,\frac{1}{2})$ and $\mathbf{p}^n = (p_1^n,\cdots,p_{\ell_n}^n) \in (\mathbb{N}^{*})^{\ell_n}$ such that $|\mathbf{p}^n| = o(n)$ and $\ell_n = o(|\mathbf{p}^n|)$. We denote by $\en$ an oriented edge chosen uniformly at random on $\partial_1 \cup \cdots \cup \partial_{\ell_n}$ such that the boundary lies to the right.
\paragraph*{Sketch of proof.}
Let us give the main ideas of the proof of Proposition~\ref{growth_rate}.
Let us write $\en$ for an oriented edge chosen uniformly at random on
$\partial_1 \cup \cdots \cup \partial_{\ell_n}$. Let $i \in \{1,\ldots,\ell_n\}$ be such that $\en \in \partial_i$.
Consider a segment $I$ of length $2a+1$ on $\partial_i$, centered at $\en$ and where $a$ can be thought of as a large constant. Using Theorem~\ref{local_limit_boundary}, we have $(\Tngp,\en) \overset{(d)}{\underset{\mathrm{loc}}{\longrightarrow}} \mathbb{H}_{\lambda(\theta)}$. Thus, we have $B_1^{\bullet}(\Tngp,I) \overset{(d)}{\approx} B_1^{\bullet}(\mathbb{H}_{\lambda(\theta)},I) $. Moreover, for $r\ge 0$ large enough, with high probability, we have $|\partial^{*}B_1^{\bullet}(\mathbb{H}_{\lambda(\theta)},I) 
\setminus \partial \mathbb{H}_{\lambda(\theta)}|
\approx 2m_{\theta}^{-1} a$. In words, this means that the growth rate of $\mathbb{H}_{\lambda(\theta)}$ is $m_{\theta}^{-1}$. Finally, covering $\partial \Tngp$ with disjoint segments
$I_1,\ldots,I_k$ and applying this estimate to each segment should yield $|\partial^{*}\Bnb| \approx m_{\theta}^{-1}|\mathbf{p}^{n}|$. Although this strategy is conceptually simple, a pathological situation
must be excluded with high probability. Indeed, it is not clear that
	\[
	\partial^{*}B_1^{\bullet}(\Tngp,I)
	\setminus \partial \Tngp
	\subset \partial^{*} \Bnb.
	\]	
	Indeed, it may happen that
	$\partial^{*}B_1^{\bullet}(\Tngp,I)$
	consists only of internal edges of $\Bnb$ (see Figure~\ref{segment_exploration}).
	Therefore, the structure of $\partial^{*}\Bnb$
	cannot be deduced directly from that of
	$\partial^{*}B_1^{\bullet}(\Tngp,I)$.
	This issue is handled in Proposition~\ref{J_lies_on_hole_hull}.
\begin{figure}[H]
\centering
\includegraphics[scale=0.5]{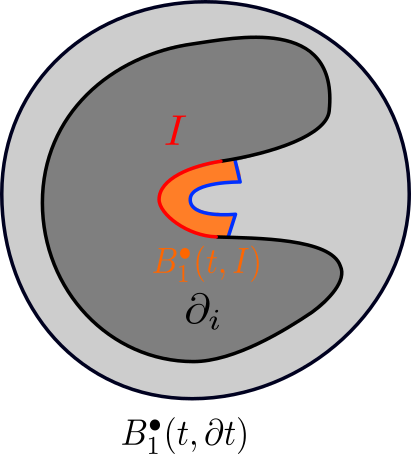}
\caption{We represent a triangulation of the $\mathbf{p}$-gon $t$, the hull $B_1^{\bullet}(t,\partial t)$ in light grey, a segment $I$ of length $2a+1$ lying on $\partial_i t$ and $B_1^{\bullet}(t,I)$ in orange. On this example, $B_1^{\bullet}(t,I) \cap \partial^{*} B_1^{\bullet}(t,\partial t) = \emptyset$, i.e. the blue segment does not lie on the hole of $ B_1^{\bullet}(t,\partial t)$.  }
\label{segment_exploration}
\end{figure}
Let us describe how Section~\ref{section_first_moment} is organised:
\begin{itemize}
	\item[$\bullet$] In Section~\ref{section_local_limit}, we use Theorem~\ref{local_limit_boundary} to show that $(\Tngp,\en)$ typically has (locally) the correct growth rate $m_{\theta}^{-1}$. We also show that a typical small neighbourhood of $(\Tngp,\en)$ behaves nicely in a way that we specify later (see Definition~\ref{definition_GOOD}).
	\item[$\bullet$] In Section~\ref{section_excluding1}, we show that the pathological situation described above does not occur with high probability.
	\item[$\bullet$] In Section~\ref{section_concluding_growth_rate}, we prove Proposition~\ref{growth_rate_first_moment}.
\end{itemize}

\subsubsection{The local convergence result}\label{section_local_limit}
Let $t \in \mathcal{T}_{n,g_n,\mathbf{p}^{n}}$ and $e$ be an oriented edge lying on $\partial_i$ for some $1 \le i \le \ell_n$. This section is dedicated to providing a notion of a 'good' neighbourhood for $(t,e)$ (Definition~\ref{definition_GOOD}). Then, using Theorem~\ref{local_limit_boundary}, we prove that $(\Tngp,\en)$ typically satisfies this condition (Proposition~\ref{proba_B_goes_to_1}).

Let us make this more precise. We denote by $x$ the starting point of $e$.
For $a \ge 0$, if $|\partial_i| \le 2a$, we set $I_a(t,e) := \partial_i$,
otherwise, $I_a(t,e)$ is defined as the segment of $\partial_i$ consisting of
edges whose endpoints are at distance at most $a$ from $x$ along $\partial_i$ (see the red segment in Figure~\ref{def_Hrte}).
We write (see Figure~\ref{def_Hrte})
\[
\Hate := B_1^{\bullet}(t,I_a(t,e)),
\qquad
\Jate := \partial^{*}\Hate \setminus \partial t.
\]

\begin{definition1}
For any $a,\delta > 0$, we say that $(t,e)$ satisfies $\mathbf{GOOD}_{-}(a,\delta)$
if $\Hate$ is a planar triangulation with one hole,
$\Jate$ forms a segment, and
\[
|\Jate| \in \big[(1-\delta)m_{\theta}^{-1}2a,(1+\delta)m_{\theta}^{-1}2a\big].
\]
\end{definition1}

\begin{figure}[H]
\centering
\includegraphics[scale=0.4]{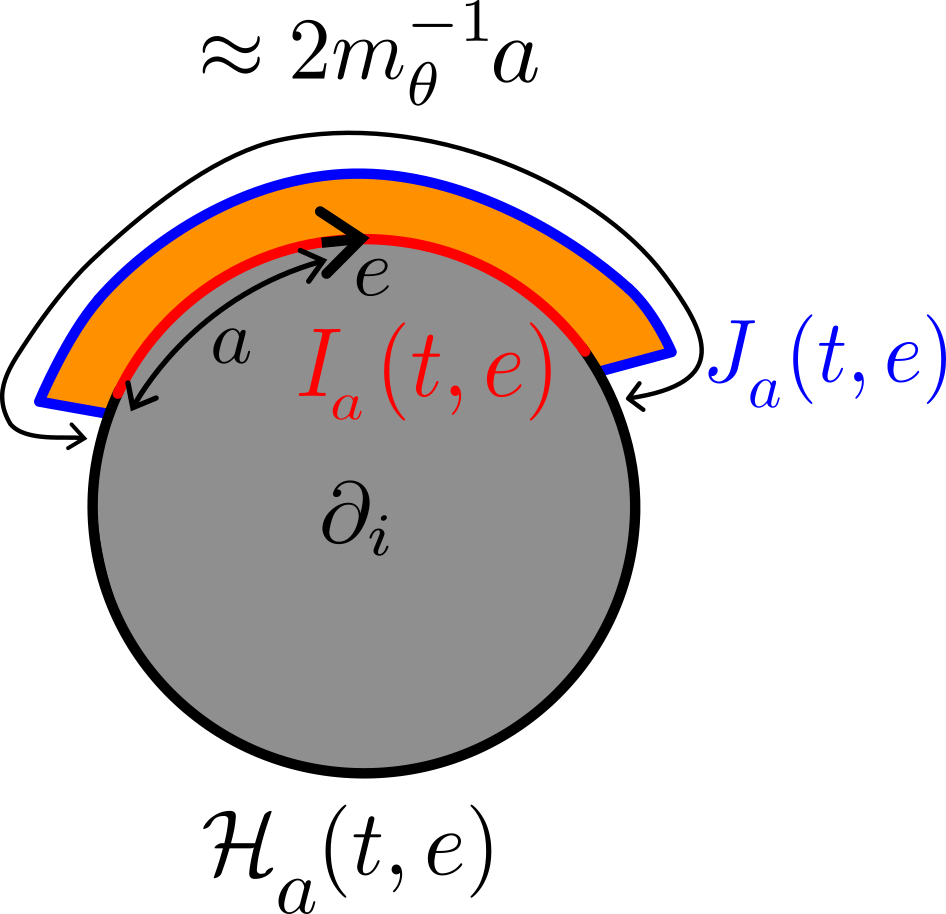}
\caption{On this example, we represent a triangulation $(t,e)$ that satisfies $\mathbf{GOOD}_{-}(a,\delta)$. }
\label{def_Hrte}
\end{figure}

Assume that $(t,e)$ satisfies $\mathbf{GOOD}_{-}(a,\delta)$. An immediate consequence is that $\Hate$ does not intersect any of the boundaries $\partial_j$ with $j \neq i$. We denote by $\Jatec$ the subset of $\Jate$ obtained by
removing the $a^{3/4}$ leftmost and rightmost edges (see the blue segment on Figure~\ref{segment_def}).

\begin{definition1}\label{definition_GOOD}
For any $a,\delta > 0$, we say that $(t,e)$ satisfies $\mathbf{GOOD}(a,\delta)$ if
the following three conditions hold:
\begin{enumerate}
	\item\label{it1}
	$(t,e)$ satisfies $\mathbf{GOOD}_{-}(a,\delta)$.
	\item\label{it2}
	Let $t' = t \setminus \Hate$.
	For any edge $e' \in \Jatec$, we have
	$|B_{a^{1/2}}^+(t',e')| \ge a^{4}$ and
	$d_{t'}(e',\partial t \cap \partial t') > a^{\frac{2}{3}}$.
	\item\label{it3}
	For all vertex $y \in \partial t$, if $d_{\partial t}(e,y) \ge a^2$, then we have
	$d_t(e,y) \ge 3a$.
\end{enumerate}
\end{definition1}
Let us comment on this definition. Items~\eqref{it2} and~\eqref{it3} are the key ingredients to prove that $\overset{\circ}{J_a}(t,e) \subset \partial^{*}B_1^{\bullet}(t,\partial t)$. See Proposition~\ref{J_lies_on_hole_hull} for details. In Item~\ref{it2} and Item~\ref{it3}, the choices for the values $a^4$, $a^{\frac{2}{3}}$ and $3a$ are arbitrary. The important point is that they satisfy the following asymptotic relations: $a^3 = o (a^4)$, $a^{\frac{2}{3}} = o(a^{\frac{3}{4}})$ and $2a+2a^{\frac{1}{2}} < 3a$.

\begin{figure}[H]
\centering
\includegraphics[scale=0.4]{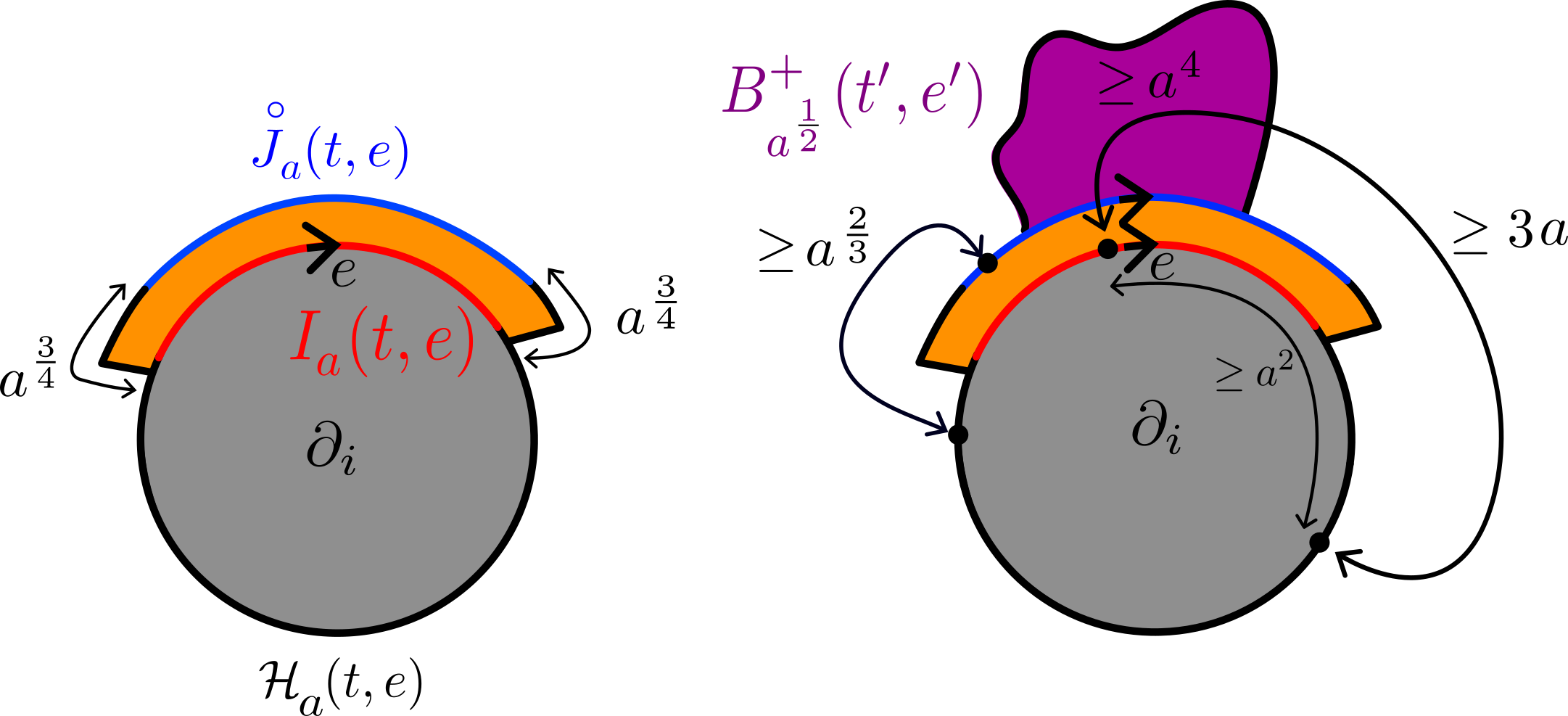}
\caption{On the left: we illustrate the segment $\Jatec$ obtained by removing the $a^{\frac{3}{4}}$ rightmost and leftmost edges from $\Jate$. On the right: we represent a triangulation $(t,e)$ that satisfies $\mathbf{GOOD}(a,\delta)$ and illustrate the different points.}
\label{segment_def}
\end{figure}

For any $a,\delta > 0$, we introduce $\mathcal{B}_n(a,\delta)$ as the event on which $(\Tngp,\en)$ satisfies $\mathbf{GOOD}_{-}(a,\delta)$. We define $\mathcal{C}_n(a,\delta)$ as the event on which $(\Tngp,\en)$ satisfies $\mathbf{GOOD}(a,\delta)$. In the next proposition, we prove that for $a$ large enough, these events occur with high probability.

\begin{proposition}\label{proba_B_goes_to_1}
For any $\delta,\eta > 0$, there exists $A \ge 0$ such that for all $a \ge A$ and for $n$ large enough,
\begin{align*}
\mathbb{P}(\mathcal{C}_n(a,\delta)) \ge 1-\eta.
\end{align*}
\end{proposition}

\begin{proof}
Fix $\delta,\eta > 0$. For any $a \ge 0$, we define $\mathcal{B}(a,\delta)$ (resp.\ $\mathcal{C}(a,\delta)$) as the event on which $(\mathbb{H}_{\lambda(\theta)},e)$ satisfies $\mathbf{GOOD}_{-}(a,\delta)$ (resp.\ $\mathbf{GOOD}(a,\delta)$), where $e$ denotes the root edge of $\mathbb{H}_{\lambda(\theta)}$. By Theorem~\ref{local_limit_boundary}, since $\mathcal{C}_n(a,\delta)$ only depends on the ball of radius $3a$ centered on $\en$, we have for any $a\ge 0$,
\begin{align*}
\mathbb{P}(\mathcal{C}_{n}(a,\delta)) \xrightarrow[n \to +\infty]{} \mathbb{P}(\mathcal{C}(a,\delta)).
\end{align*}
Therefore, it is sufficient to prove that for $a$ large enough,
\begin{align*}
\mathbb{P}(\mathcal{C}(a,\delta)) \ge 1 -\eta.
\end{align*}
For $i \ge 0$, let $x_{i}$ (resp.\ $x_{-i}$) denote the $i$-th vertex to the right (resp.\ to the left) of the starting point of $e$ on $\partial \mathbb{H}_{\lambda(\theta)}$. We use the shorthand notation $I_a := I_a(\mathbb{H}_{\lambda(\theta)},e)$ and $J_a := J_a(\mathbb{H}_{\lambda(\theta)},e)$. We claim that it suffices to prove that there exist constants $c_{\theta},C_{\theta} > 0$ such that for $a$ large enough,
\begin{align}\label{equa_to_prove}
&\mathbb{P}(\mathcal{B}(a,\delta)) \ge 1-\frac{\eta}{2} \nonumber\\
&\mathbb{P}\big(|B_{a^{\frac{1}{2}}}^{+}(\mathbb{H}_{\lambda(\theta)},e)| \ge a^4\big) \ge 1 -\exp(-c_{\theta}a^{\frac{1}{2}}),\\
&\nonumber \mathbb{P}\big(d_{\mathbb{H}_{\lambda(\theta)}}(x_0,x_i) \le C_{\theta}i\big) \ge 1-C_{\theta}\exp(-c_{\theta}i) \hspace*{0.3cm} \forall i \ge 0.
\end{align}

 Indeed, under $\mathcal{B}(a,\delta)$, let
\[
H' = \mathbb{H}_{\lambda(\theta)}\setminus B_1^{\bullet}(\mathbb{H}_{\lambda(\theta)},I_a).
\]
Since $\mathcal{B}(a,\delta)$ is measurable with respect to $B_1^{\bullet}(\mathbb{H}_{\lambda(\theta)},I_a)$ and by the spatial Markov property of $\mathbb{H}_{\lambda(\theta)}$, conditionally on $\mathcal{B}(a,\delta)$, we have $H' \overset{(d)}{=} \mathbb{H}_{\lambda(\theta)}$. For $e' \in \overset{\circ}{J_a}$ and $i \ge 0$, let $y_i$ (resp.\ $y_{-i}$) denote the $i$-th vertex to the right (resp.\ to the left) of the starting point of $e'$ along $\partial H'$. Define the event
\begin{align*}
\mathcal{D}(e') := \{|B_{a^{\frac{1}{2}}}^{+}(H',e')| \ge a^4\}
\cap \{\forall i \in \mathbb{Z}\setminus \{-a^{3/4},\ldots,a^{3/4}\},\ d_{H'}(y_0,y_{i}) > a^{\frac{2}{3}}\}.
\end{align*}
Then, under
\[
\mathcal{B}(a,\delta)\cap \{\forall i \in \mathbb{Z}\setminus \{-a^2,\ldots,a^2\},\ d_{\mathbb{H}_{\lambda(\theta)}}(x_0,x_i) \ge 3a\}
\cap \bigcap_{e' \in \overset{\circ}{J_a}}\mathcal{D}(e'),
\]
the event $\mathcal{C}(a,\delta)$ occurs. Using~\eqref{equa_to_prove}, and the bound $|\Jatec| \le m_{\theta}^{-1}2a(1+\delta)$, a direct computation shows that for $a$ large enough,
\begin{align*}
\mathbb{P}(\mathcal{C}(a,\delta)) \ge 1-\eta.
\end{align*}

 We now prove~\eqref{equa_to_prove}. The value of $C_{\theta},c_{\theta}$ may change from line to line. We start with the first inequality and use a peeling exploration. We define a peeling algorithm $\mathcal{A}_a$. Let $t$ be a triangulation of the half-plane with one infinite hole and root edge $e$. If $I_a\cap \partial^{*}t$ is non-empty, let $e'$ be the leftmost edge in this set, and define $\mathcal{A}_a(t)$ as the first edge to the left of $e'$ along $\partial^{*}t$. Otherwise, define $\mathcal{A}_a(t)$ arbitrarily. Let $(\overline{\mathcal{E}}^{\mathcal{A}_a}_k(\mathbb{H}_{\lambda(\theta)}))_{k \ge 0}$ be the associated peeling exploration, and set $\overline{t}_k = \overline{\mathcal{E}}^{\mathcal{A}_a}_k(\mathbb{H}_{\lambda(\theta)})$.

For $k,i \ge 0$, we recall the events $\mathbf{R}^i_k$, $\mathbf{L}^i_k$ and $\mathbf{C}_k$ defined in Section~\ref{peeling_exploration}. We define
\begin{align}\label{def_tau}
\tau_a = \inf \Big\{k \ge 0 : I_a \cap \partial^{*}\overline{t}_k = \emptyset\Big\}.
\end{align}
It is straightforward to check that $\tau_a$ is almost surely finite and that $\overline{t}_{\tau_a} = B_1^{\bullet}(\mathbb{H}_{\lambda(\theta)},I_a)$. For $k \ge 0$, let $M_k$ denote the number of edges of $I_a$ still exposed on the hole and we denote by $R_k$ (resp. $L_k$) the number of edges swallowed to the right (resp. the left) by the $k$-th discovered triangle. The sequences $(R_k)_{k \ge 0}$ and $(L_k)_{k \ge 0}$ are i.i.d. by the spatial Markov property of $\mathbb{H}_{\lambda(\theta)}$. Then,
\[
M_k = \max\big(0,2a+1-R_1-\cdots-R_k\big).
\]
In particular,
\begin{align*}
\tau_a = \inf \{k \ge 0 : R_1+\cdots+R_k \ge 2a+1 \}.
\end{align*}

Using \eqref{peeling_transitions} and \eqref{formula_wlambda}, one checks that $R_k$ has an exponentially decaying tail and
\[
\mathbb{E}[R_k] = \frac{1}{2}\left(\frac{1}{\sqrt{1+8h}}-\frac{\sqrt{1-4h}}{\sqrt{1+8h}}\right).
\]
By standard large deviation estimates, there exists $c_{\theta,\delta} > 0$ such that for $a$ large enough,
\begin{align}\label{esti_taur}
\mathbb{P}\Bigg(\tau_a \in \Big[\Big(1-\frac{\delta}{2}\Big)\frac{2a}{\mathbb{E}[R_1]},\ \Big(1+\frac{\delta}{2}\Big)\frac{2a}{\mathbb{E}[R_1]}\Big]\Bigg)
\ge 1-\exp(-c_{\theta,\delta}a).
\end{align}

For all $k \ge 1$, let $X_k = |\partial^{*}\overline{t}_k \setminus \partial \overline{t}_k|$ and $Y_k = |\partial\overline{t}_k \setminus \partial^{*}\overline{t}_k|$ (see Figure~\ref{peeling_XY}). Note that
\[
X_{\tau_a} = |\partial^{*} B_1^{\bullet}(\mathbb{H}_{\lambda(\theta)},I_a) \setminus \partial \mathbb{H}_{\lambda(\theta)}|
= |J_a|.
\]

\begin{figure}[H]
\centering
\includegraphics[scale=0.3]{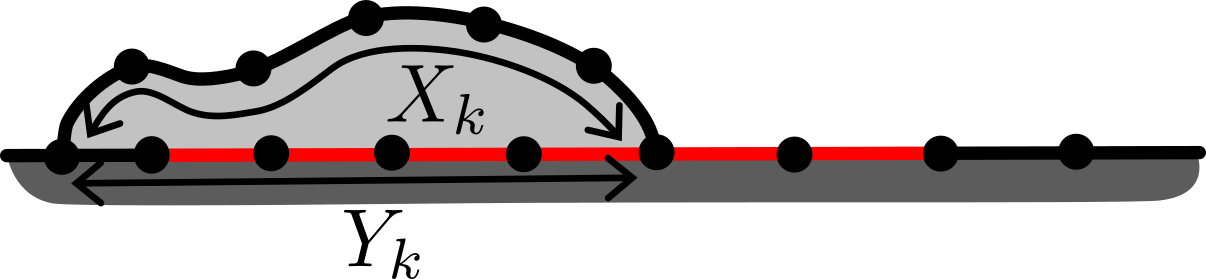}
\caption{The triangulation with holes $\overline{t}_k$ for $k \in \{0,\ldots,\tau_a-1\}$. The variable $X_k$ counts edges exposed on the hole that do not belong to the boundary, while $Y_k$ counts boundary edges that have been swallowed.}
\label{peeling_XY}
\end{figure}

Since $J_a$ is almost surely a segment, it suffices to show
\[
\mathbb{P}\Big(X_{\tau_a} \in [(1-\delta)m_{\theta}^{-1}2a,(1+\delta)m_{\theta}^{-1}2a]\Big)
\ge 1-\frac{\eta}{2}.
\]

Let $S_k = X_k - Y_k$ denote the evolution of the perimeter process. Its increments $\Delta S_k := S_{k+1}-S_k$ are i.i.d.\ with exponentially decaying tails, and by \eqref{peeling_transitions} and \eqref{formula_W},
\[
\mathbb{E}[\Delta S_k] = \frac{\sqrt{1-4h}}{\sqrt{1+8h}}.
\]
Combining this with~\eqref{esti_taur} and standard large deviations again yields
\begin{align}\label{estimate_S}
\mathbb{P}\Bigg(S_{\tau_a} \in \Big[(1-\delta)\frac{\mathbb{E}[\Delta S_1]}{\mathbb{E}[R_1]}2a,\ (1+\delta)\frac{\mathbb{E}[\Delta S_1]}{\mathbb{E}[R_1]}2a\Big]\Bigg)
\ge 1-\exp(-c_{\theta,\delta}a).
\end{align}

At any step $k \in \{0,\ldots,\tau_a-1\}$, the set $\partial^{*}\overline{t}_k \setminus \partial \overline{t}_k$ consists of a single segment on the hole, and the algorithm always peels the rightmost edge of this segment. Consequently,
\begin{align*}
&Y_{k+1}-Y_k = \max(R_k,L_k-X_k),\\
&X_{k+1}-X_k =
\begin{cases}
0 & \text{if } \mathbf{R}_k^i \text{ occurs},\\
\max(-X_k+1,-L_k) & \text{if } \mathbf{L}_k^i \text{ occurs},\\
1 & \text{if } \mathbf{C}_k \text{ occurs}.
\end{cases}
\end{align*}
It follows that
\[
X_k \ge \sum_{j=0}^{k-1}\big(\mathbf{1}_{\mathbf{C}_j}-L_j\mathbf{1}_{\cup_{i \ge 0}\mathbf{L}_j^{i}}\big).
\]
The summands are i.i.d., have a positive drift since they dominate $\Delta S_k$, and exponentially decaying tails. Hence, with high probability,
\[
\mathbb{P}(L_k-X_k \ge 0) \le \exp(-c_{\theta}k).
\]
Using that $L_k-X_k \le 0$ implies $Y_{k+1}-Y_k = R_k$ and \eqref{esti_taur}, one obtains
\begin{align}\label{estimate_Y}
\mathbb{P}\big(Y_{\tau_a} \in [(1-\delta)2a,(1+\delta)2a]\big) \ge 1-\exp(-c_{\theta,\delta}a).
\end{align}
Since $|J_a| = X_{\tau_a} = S_{\tau_a}+Y_{\tau_a}$, we conclude that
\begin{align*}
\mathbb{P}\bigg(|J_a(e)| \in \bigg[(1-\delta)\bigg(1+\frac{\mathbb{E}[\Delta S_1]}{\mathbb{E}[R_1]}
\bigg)2a ,(1+\delta)\bigg(1+\frac{\mathbb{E}[\Delta S_1]}{\mathbb{E}[R_1]}
\bigg)2a \bigg]\bigg)
\ge 1-\exp(-c_{\theta,\delta}a).
\end{align*}
Noting that
\[
1+\frac{\mathbb{E}[\Delta S_1]}{\mathbb{E}[R_1]}
= \frac{1-2h_{\theta}+\sqrt{1-4h_{\theta}}}{2h_{\theta}}
= m_{\theta}^{-1},
\]
this completes the proof of the first inequality in \eqref{equa_to_prove}.

We now turn to the two last inequalities in \eqref{equa_to_prove}. These estimates are proved in~\cite[Lemma~4.6 and Proposition~4.10]{Ray_geometry_halfplane} for type-II half-plane triangulations. Although we consider type-I triangulations, the arguments are the same, so we only sketch the proof.

The main tool is a peeling exploration with an algorithm $\mathcal{A}_{\mathrm{metric}}$ that reveals metric hulls. Let $\rho$ denote the starting vertex of $e$. For a triangulation of the half-plane $t$ with one infinite hole, let $x$ be the leftmost boundary vertex minimizing the distance to $\rho$, and define $\mathcal{A}_{\mathrm{metric}}(t)$ as the edge on the hole immediately to the left of $x$. Using the same notation as above, define for $r \ge 0$,
\begin{align*}
\tau_r = \inf \Big\{k \ge 0 : \min_{x \in \partial^{*}\overline{t}_k} d_{\overline{t}_k}(x,\rho) \ge r\Big\}.
\end{align*}
Then $\overline{t}_{\tau_r} = B_r^{\bullet}(\mathbb{H}_{\lambda(\theta)},e)$ and since each peeling step adds at least one triangle, we have $|B_r^{+}(\mathbb{H}_{\lambda(\theta)},e)| \ge \tau_r$. Between $\tau_r$ and $\tau_{r+1}$, the exploration peels all boundary vertices at distance $r$ until they are swallowed. The number of boundary vertices $x_i$ to the right of the root that belong to $B_r^{\bullet}(\mathbb{H}_{\lambda(\theta)},e)$ is bounded by $\sum_{j=1}^{r} R_{\tau_j-1}$. Since $(R_k)_{k \ge 0}$ are i.i.d.\ with exponentially decaying tails, there exist $c_{\theta},C_{\theta}>0$ such that
\begin{align*}
\mathbb{P}(d_{\mathbb{H}_{\lambda(\theta)}}(x_0,x_i) \le C_{\theta}i ) \ge 1-C_{\theta} \exp(-c_{\theta}i).
\end{align*}
For any $\ell \ge 1$, using \eqref{estimate_S}, we have
\[
\mathbb{P}(\tau_{r+1}-\tau_r \le c_{\theta}\ell \mid X_{\tau_r} = \ell) \le \exp(-c_{\theta}\ell).
\] 
Using the positive drift of $(S_n)_{n \ge 0}$, this implies
\[
\mathbb{P}( X_{\tau_{r+1}} \le (1+c_{\theta})\ell \mid X_{\tau_r} = \ell) \le \exp(-c_{\theta}\ell).
\]
Putting together the two last estimates, there exist $c_{\theta},C_{\theta}>0$ such that 
\[
\mathbb{P}( \tau_r \le \exp(c_{\theta}r)) \le C_{\theta}\exp(-c_{\theta}r).
\]
Using this with the fact that $|B_r^{+}(\mathbb{H}_{\lambda(\theta)},e)| \ge \tau_r$ and replacing $r$ with $a^{\frac{1}{2}}$ completes the proof of~\eqref{equa_to_prove} and of the proposition.
\end{proof}

\subsubsection{Hulls can be discovered locally}\label{section_excluding1}
This section is dedicated to proving the following proposition.
\begin{proposition}\label{J_lies_on_hole_hull}
For any $\delta>0$, there exists $a_0 \ge 0$ such that for all $a \ge a_0$ and for $n$ large enough,
\begin{align*}
\mathbb{P}\bigg(
\frac{|\partial^{*}\Bnb|}{|\Bnb|}\ge \ct \hspace*{0.1cm} \text{ and }\hspace*{0.1cm} \overset{\circ}{J_a}(\Tngp,\en)\nsubseteq \partial^{*}\Bnb
\bigg)
\le \delta .
\end{align*}
\end{proposition}
For $t \in \mathcal{T}_{\mathbf{p}^n}(n,g_n)$ and $e$ an oriented edge on $\partial t$, we recall the notation
\[
\Hate := B_1^{\bullet}(t,I_a(t,e))
\quad\text{and}\quad
\Jate = \partial^{*}\mathcal{H}_{a}(t,e) \setminus \partial t.
\]
On the event $\mathcal{C}_n(a,\delta)$, the subset $\overset{\circ}{J_a}(t,e)$ is obtained from $J_a(t,e)$ by removing the $a^{3/4}$ leftmost and rightmost edges. See Figure~\ref{segment_def} for an illustration. The proof proceeds as follows:
\begin{enumerate}
	\item  We show that, if $\overset{\circ}{J_a}(t,e) \nsubseteq \partial^{*} B_1^{\bullet}(t,\partial t)$, then $B_1^{\bullet}(t,\partial t)$ must contain a region of diameter at most $2a^{\frac{1}{2}}$ with size at least $a^4$. This is established in Proposition~\ref{contains_ball}. Note that this proposition is entirely deterministic.
	\item Then, if $\overset{\circ}{J_a}(t,e) \nsubseteq \partial^{*} B_1^{\bullet}(t,\partial t)$ for many $e \in \partial t$, we can construct many disjoint large subsets contained in $ B_1^{\bullet}(t,\partial t)$ and we obtain a contradiction with the bound on $\Bnb$ given by Proposition~\ref{second_moment_bounded_degree}. 
\end{enumerate}

\begin{proposition}\label{contains_ball}
Fix $t \in \mathcal{T}_{\mathbf{p}^n}(n,g_n)$ and an oriented edge $e$ on $\partial t$. Let $a \ge 1$ and $\delta > 0$, and assume that $(t,e)$ satisfies $\mathbf{GOOD}(a,\delta)$. 
Assume $\overset{\circ}{J_a}(t,e) \nsubseteq \partial^{*}B_1^{\bullet}(t,\partial t)$. Let us fix $u \in \overset{\circ}{J_a}(t,e) \setminus \partial^{*} B_1^{\bullet}(t,\partial t)$ arbitrarily and write $t' = t \setminus \mathcal{H}_a(t,e)$. Then
\begin{align*}
\mathcal{H}_a(t,e)\cup B_{a^{1/2}}^+(t',u) \subset B_1^{\bullet}(t,\partial t).
\end{align*}
\end{proposition}

\begin{proof}

We recall that $t \setminus B_1^+(t,\partial t)$ is a family of connected components $t_1,\dots,t_{k}$ and $B_1^{\bullet}(t,\partial t)$ is obtained by gluing to $B_1^+(t,\partial t)$ the components $t_i$ with at most $n$ internal triangles. We claim that it suffices to prove that there exists $1 \le i \le k$ such that $ B_{a^{1/2}}^+(t',u) \subset t_i$ to conclude the proof. Indeed, in that case, $t_i$ has less than $n$ internal triangles since $u \in t_i$ and $u \notin \partial^{*} B_1^{\bullet}(t,\partial t)$. Let us now prove this claim. We recall that for $t_0$ a triangulation with holes, we denote by $F_{\mathrm{in}}(t_0)$ the set of internal faces of $t_0$. It suffices to prove $F_{\mathrm{in}}(B_{a^{1/2}}^+(t',u)) \cap F_{\mathrm{in}}(B_1^+(t,\partial t)) = \emptyset$. Let us reason by contradiction and assume this intersection is non-empty and let $f \in F_{\mathrm{in}}(B_{a^{1/2}}^+(t',u)) \cap F_{\mathrm{in}}(B_1^+(t,\partial t))$. Then $f$ is incident to a vertex $v \in \partial t$. Moreover, since $f \in F_{\mathrm{in}}(B_{a^{1/2}}^+(t',u))$, we have $v \in \partial t'$ and $d_{t'}(u,v) \le a^{\frac{1}{2}}$. Thus, $v \in \partial t \cap \partial t'$. These last two facts contradict the second condition in the definition of $\mathbf{GOOD}(a,\delta)$ applied to $(t,e)$, since $d_{t'}(u,\partial t \cap \partial t') > a^{2/3} > a^{1/2}$. This concludes the proof.
\end{proof}

Finally, we state and prove Proposition~\ref{J_lies_on_hole_hull}.

\begin{proof}
We reason by contradiction. Fix $\delta>0$ and assume that the conclusion does not hold. Using Proposition~\ref{second_moment_bounded_degree} and Markov's inequality, for any $a > 0$ and for $n$ large enough,
\begin{align*}
\mathbb{P}\bigg(\frac{|\partial^{*}\Bnb|}{|\mathbf{p}^{n}|}\ge a\bigg)
\le C_{\theta} a^{-2}.
\end{align*}
We choose $A > 16\delta^{-1}\ct^{-1}$ such that for $a\ge A $, the right-hand side above is at most $\delta/4$. Using Proposition~\ref{proba_B_goes_to_1}, up to increasing $A$, we may also assume that for $n$ large enough we have $\mathbb{P}(\mathcal{C}_n(a,\delta)) \ge 1 - \frac{\delta}{4}$. Then, for arbitrarily large values of $a \ge 1$ and $n$, we have 
\begin{align}\label{proba_lower_bound}
\mathbb{P}\bigg(
\mathcal{C}_n(a,\delta),\
|\Bnb|\le \ct^{-1} a|\mathbf{p}^{n}|,\
\overset{\circ}{J_a}(\Tngp,\en)\nsubseteq \partial^{*}\Bnb
\bigg)
\ge \frac{\delta}{4}.
\end{align}

Since $\en$ is uniform on $\partial \Tngp$, it follows that there exists
$t\in \mathcal{T}_{\mathbf{p}^{n}}(n,g_n)$ such that
$|B_1^{\bullet}(t,\partial t)|\le \ct^{-1} a|\mathbf{p}^{n}|$ and there are at least $\frac{\delta}{4}|\mathbf{p}^{n}|$ distinct oriented edges $e_1,\dots,e_{\frac{\delta}{4}|\mathbf{p}^{n}|}$ on $\partial t$ for which $(t,e_i)$ satisfies $\mathbf{GOOD}(a,\delta)$ but
\[
\overset{\circ}{J_a}(t,e_i)\nsubseteq \partial^{*}B_1^{\bullet}(t,\partial t).
\]
We fix such a triangulation $t$. From now on, the argument is entirely deterministic, and we show that this leads to a contradiction. Up to relabelling the $e_i$, we may assume that for $1 \le i < j \le \frac{\delta}{16 a^2}|\mathbf{p}^n|$,
\[
d_{\partial t}(e_i,e_j)\ge a^2.
\]
Fix $i\in\{1,\dots,\frac{\delta}{16a^2}|\mathbf{p}^{n}|\}$. Since
$\overset{\circ}{J_a}(t,e_i)\nsubseteq \partial^{*}B_1^{\bullet}(t,\partial t)$,
there exists an edge
$u_i\in \overset{\circ}{J_a}(t,e_i)$ such that
$u_i\notin \partial^{*}B_1^{\bullet}(t,\partial t)$.
We define $
t_i := t\setminus \mathcal{H}_a(t,e_i)$. By Proposition~\ref{contains_ball}, we have
\begin{align}\label{balls_contained_hull}
 B_{a^{1/2}}^+(t_i,u_i)
\subset B_1^{\bullet}(t,\partial t).
\end{align}
Using the fact that for
$1\le i<j\le \frac{\delta}{16 a^2}|\mathbf{p}^{n}|$
we have $d_{\partial t}(e_i,e_j)\ge a^2$, together with the assumption that $(t,e_i)$ satisfies $\mathbf{GOOD}(a,\delta)$ (more precisely Item $3$ of Definition~\ref{definition_GOOD}), we deduce that
$d_t(e_i,e_j)\ge 3a$.
Moreover,
$d_t(e_i,u_i)\le a+1$ and $d_t(e_j,u_j)\le a+1$.
It follows that 
\begin{align}\label{balls_do_not_intersect2}
V(B_{a^{1/2}}^+(t_i,u_i))\cap V(B_{a^{1/2}}^+(t_j,u_j))=\emptyset,
\end{align}
where $V(\cdot)$ denotes the set of vertices. In particular, they do not share any internal face. Moreover, we recall that since $(t,e_i)$ satisfies $\mathbf{GOOD}(a,\delta)$, we have:
\begin{align}\label{large_balls}
\forall i \in \bigg\{1,\cdots,\frac{\delta}{16 a^2}|\mathbf{p}^n| \bigg\}\text{,}\hspace*{0.2cm}|B_{a^{1/2}}^+(t_i,u_i)| \ge a^4.
\end{align}

Combining \eqref{balls_contained_hull}, \eqref{balls_do_not_intersect2} and \eqref{large_balls}, we obtain
\begin{align*}
|B_1^{\bullet}(t,\partial t)|
&\ge
\sum_{i=1}^{\frac{\delta}{16 a^2}|\mathbf{p}^{n}|}
\big|B_{a^{1/2}}^+(t_i,u_i)\big|
\ge
a^{4}\,\frac{\delta}{16a^2}|\mathbf{p}^{n}|
=
\frac{\delta a^2}{16}|\mathbf{p}^{n}|
>
\ct^{-1} a|\mathbf{p}^{n}|,
\end{align*}
where in the last inequality we have used the fact that $a \ge A \ge 16 \delta^{-1}\ct^{-1}$. This contradicts the assumption
$|B_1^{\bullet}(t,\partial t)|\le \ct^{-1} a|\mathbf{p}^{n}|$
and concludes the proof.
\end{proof}

\subsubsection{Proof of the $L^1$ convergence}\label{section_concluding_growth_rate}
This section is dedicated to proving Proposition~\ref{growth_rate}. We start with a deterministic lemma providing a relation between
$|\partial^{*}B_1^{\bullet}(t,\partial t)|$
and
$|\overset{\circ}{J_a}(t,e)|$.

\begin{lemma}\label{inequalities}
For any $t\in \mathcal{T}_{\mathbf{p}^n}(n,g_n)$ and any $\delta\ge 0$, for $a$ large enough, we have
\begin{align*}
(1-3\delta)m_{\theta}^{-1}
\#\Big\{ e\in \partial t \ :\ &
\overset{\circ}{J_a}(t,e)\subset \partial^{*}B_1^{\bullet}(t,\partial t)
\ \text{and } (t,e)\text{ satisfies }\mathbf{GOOD}(a,\delta)
\Big\} \\
&\le |\partial^{*}B_1^{\bullet}(t,\partial t)| \\
&\le \sum_{e\in \partial t}\frac{|J_a(t,e)|}{|I_a(t,e)|}.
\end{align*}
\end{lemma}

\begin{proof}
Fix $t\in \mathcal{T}_{\mathbf{p}^n}(n,g_n)$ and $a,\delta>0$.
To any edge $e'\in \partial^{*}B_1^{\bullet}(t,\partial t)$, we associate $\varphi(e')$, defined as the third vertex of the triangle incident to $e'$ and contained in $B_1^{\bullet}(t,\partial t)$. By definition, $\varphi(e')\in \partial t$.
We also define $\psi(e')$ as the oriented edge on $\partial t$ with starting point $\varphi(e')$ and such that $\partial t$ lies to the right of $\psi(e')$. Then, for all $e\in \partial t$ and $e'\in \partial^{*}B_1^{\bullet}(t,\partial t)$, we have
\begin{align}\label{equivalence_equa}
e'\in J_a(t,e)
\iff
\psi(e')\in I_a(t,e)
\iff
e\in I_a(t,\psi(e'))
\qquad
\text{(see Figure~\ref{equivalence}).}
\end{align}

\begin{figure}[H]
\centering
\includegraphics[scale=0.4]{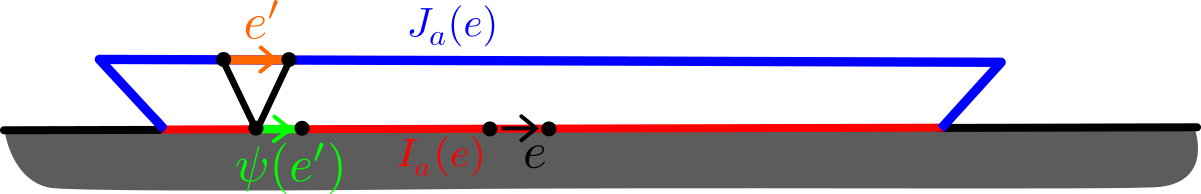}
\caption{In this example, $e'\in J_a(t,e)$. Hence $\psi(e')\in I_a(t,e)$, or equivalently $e\in I_a(t,\psi(e'))$.}
\label{equivalence}
\end{figure}

For $e\in \partial t$, we define $X(e) := J_a(t,e) \cap \partial^{*}B_1^{\bullet}(t,\partial t)$. Using \eqref{equivalence_equa}, we see that for each
$e'\in \partial^{*}B_1^{\bullet}(t,\partial t)$,
the edge $e'$ belongs exactly to the sets $X(e)$ for which $e\in I_a(t,\psi(e'))$.
Since $|I_a(t,\psi(e'))|\le 2a+1$, we obtain
\begin{align} \label{bound_to_conclude}
\sum_{e\in \partial t}\frac{|X(e)|}{2a+1}
\le
|\partial^{*}B_1^{\bullet}(t,\partial t)|
\le
\sum_{e\in \partial t}\frac{|J_a(t,e)|}{|I_a(t,e)|},
\end{align}
where the second inequality can be obtained observing that for each boundary $\partial_i t$, the numbers $|I_a(t,e)|$ are constant for $e \in \partial_i t$ and thus $\partial_i t$ contributes at most $\somme{e \in\partial_i t }{}{}\frac{|J_a(t,e)|}{|I_a(t,e)|}$ edges lying on $\partial^{*}B_1^{\bullet}(t,\partial t)$, then we sum over $i$.

Recall that $\overset{\circ}{J_a}(t,e)$ is obtained from $J_a(t,e)$ by removing the $a^{3/4}$ leftmost and rightmost edges. Also recall the definition of $\mathbf{GOOD}(a,\delta)$ in Definition~\ref{definition_GOOD}. It follows that if $(t,e)$ satisfies $\mathbf{GOOD}(a,\delta)$, then we have
\[
|\overset{\circ}{J_a}(t,e)|
\ge (1-\delta)m_{\theta}^{-1}2a - 2a^{3/4}.
\]
Choosing $a$ sufficiently large so that $a^{3/4}\le \frac{\delta}{2} m_{\theta}^{-1} a$, we deduce that
\[
|X(e)| \ge (1-2\delta)m_{\theta}^{-1}2a\,
\indi{\overset{\circ}{J_a}(t,e)\subset \partial^{*}B_1^{\bullet}(t,\partial t)}
\indi{(t,e)\text{ satisfies }\mathbf{GOOD}(a,\delta)}
.
\]
Combining this with the lower bound in \eqref{bound_to_conclude}, and choosing $a$ large enough such that $(1-3\delta) \le (1-2\delta)\frac{2a}{2a+1}$ concludes the proof.
\end{proof}

We can now give the proof of Proposition~\ref{growth_rate}.

\begin{proof}
Fix $\delta>0$. We first prove that
\begin{align}\label{upper_bound}
\mathbb{P}\bigg(
|\partial^{*}\Bnb|
\ge (1+\delta)m_{\theta}^{-1}|\mathbf{p}^{n}|
\bigg)
\underset{n\to\infty}{\longrightarrow} 0.
\end{align}

Using Lemma~\ref{inequalities}, there exists $a > 0$ such that for $n$ large enough, we have
\begin{align}\label{first_eq}
\frac{|\partial^{*}\Bnb|}{|\mathbf{p}^{n}|}
\le
\frac{1}{|\mathbf{p}^{n}|}
\sum_{e\in \partial \Tngp}
\frac{|J_a(e)|}{|I_a(e)|}.
\end{align}
The right-hand side is bounded above by
\begin{align}\label{new_equa_thomas}
&(1+\delta/4)m_{\theta}^{-1}
+
\frac{1}{|\mathbf{p}^{n}|}
\sum_{e\in \partial \Tngp}
\frac{|J_a(e)|}{|I_a(e)|}
\indi{
\frac{|J_a(e)|}{|I_a(e)|}
\ge (1+\delta/4)m_{\theta}^{-1}
}.
\end{align}

We now bound the second term in $L^1$.
By the Cauchy--Schwarz inequality,
\begin{align}\label{sec_eq}
&\mathbb{E}\bigg[
\frac{1}{|\mathbf{p}^{n}|}
\sum_{e\in \partial \Tngp}
\frac{|J_a(e)|}{|I_a(e)|}
\indi{
\frac{|J_a(e)|}{|I_a(e)|}
\ge (1+\delta/4)m_{\theta}^{-1}
}
\bigg] \\
&=
\mathbb{E}\bigg[
\frac{|J_a(\en)|}{|I_a(\en)|}
\indi{
\frac{|J_a(\en)|}{|I_a(\en)|}
\ge (1+\delta/4)m_{\theta}^{-1}
}
\bigg] \nonumber\\
&\le
\mathbb{E}\bigg[
\bigg(\frac{|J_a(\en)|}{|I_a(\en)|}\bigg)^2
\bigg]^{1/2}
\mathbb{P}\bigg(
\frac{|J_a(\en)|}{|I_a(\en)|}
\ge (1+\delta/4)m_{\theta}^{-1}
\bigg)^{1/2}. \nonumber
\end{align}

For a vertex $x$ of $\Tngp$, recall that
$D_x=\deg_{\Tngp}(x)$.
For $e\in \partial \Tngp$, denote by $\rho(e)$ its starting point.
We claim that $|J_a(\en)|
\le
\sum_{e\in I_a(\en)} D_{\rho(e)}$. Indeed, to each $e \in \Jate$, we can associate the opposite corner incident to a vertex in $I_a(\en)$ which gives the desired bound. Together with Proposition~\ref{second_moment_bounded_degree},
we obtain that for $n$ large enough,
\begin{align}\label{third_eq}
\mathbb{E}\bigg[
\bigg(\frac{|J_a(\en)|}{|I_a(\en)|}\bigg)^2
\bigg]
\le C_{\theta},
\end{align}
for some $ C_{\theta} > 0$ that only depends on the parameter $\theta$. Fix $\eta>0$. By Proposition~\ref{proba_B_goes_to_1}, there exists $a \ge 0$ such that,
for $n$ large enough,
\begin{align}\label{fourth_eq}
\mathbb{P}\bigg(
\frac{|J_a(\en)|}{|I_a(\en)|}
\ge (1+\delta/4)m_{\theta}^{-1}
\bigg)
\le
\mathbb{P}\bigg(
\mathcal{C}_n\bigg(a,\frac{\delta}{4}\bigg)^c
\bigg)
\le \eta.
\end{align}
This bounds \eqref{sec_eq} by $C_{\theta}^{\frac{1}{2}} \eta^{\frac{1}{2}}$.
Putting this with \eqref{new_equa_thomas} and Markov's inequality, we deduce that,
for $n$ large enough,
\begin{align*}
\mathbb{P}\bigg(
|\partial^{*}\Bnb|
\ge (1+\delta)m_{\theta}^{-1}|\mathbf{p}^{n}|
\bigg)
&\le
\frac{4}{3\delta}
m_{\theta}
C_{\theta}^{1/2}
\eta^{1/2}.
\end{align*}
Letting $\eta\to 0$ concludes the proof of \eqref{upper_bound}.

We now prove that, for $n$ large enough,
\begin{align*}
\mathbb{P}\bigg(
\frac{|\partial^{*}\Bnb|}{|\Bnb|}
\ge \ct,
\;
|\partial^{*}\Bnb|
\le (1-4\delta)m_{\theta}^{-1}|\mathbf{p}^{n}|
\bigg)
\underset{n\to\infty}{\longrightarrow} 0.
\end{align*}
Indeed, proving the result with $4\delta$ instead of $\delta$ is sufficient to conclude the proof.  Fix $\eta > 0$. Assume that
$\frac{|\partial^{*}\Bnb|}{|\Bnb|}\ge \ct$
and
$|\partial^{*}\Bnb|
\le (1-4\delta)m_{\theta}^{-1}|\mathbf{p}^{n}|$. For any $a \ge 0$, we define
\[
A_n(a,\delta)
:=
\#\Big\{
e\in \partial \Tngp :
\overset{\circ}{J_a}(e)\nsubseteq \partial^{*}\Bnb
\ \text{or } (\Tngp,e)
\text{ does not satisfy }\mathbf{GOOD}(a,\delta/4)
\Big\}.
\]
Using the lower bound in Lemma~\ref{inequalities},
for $a$ large enough we have
\begin{align}\label{lower_bound_ineq}
|\partial^{*}\Bnb| \ge
(1-3\delta)m_{\theta}^{-1}(|\mathbf{p}^n|-A_n(a,\delta)).
\end{align}

Under $|\partial^{*}\Bnb| \le (1-4\delta) m_{\theta}^{-1}|\mathbf{p}^n|$, we get $(1-3\delta)m_{\theta}^{-1}A_n \ge (1-3\delta)m_{\theta}^{-1}|\mathbf{p}^n| - (1-4\delta)m_{\theta}^{-1}|\mathbf{p}^n| = \delta m_{\theta}^{-1}|\mathbf{p}^n|$, so $A_n \ge \frac{\delta}{1-3\delta}|\mathbf{p}^n| \ge \delta |\mathbf{p}^n|$. On the other hand,
\begin{align*}
\frac{1}{|\mathbf{p}^{n}|}
\mathbb{E}\bigg[
\indi{
\frac{|\partial^{*}\Bnb|}{|\Bnb|}
\ge \ct
}
A_n(a,\delta)
\bigg]
&\le
\mathbb{P}\bigg(
\frac{|\partial^{*}\Bnb|}{|\Bnb|}
\ge \ct,
\mathcal{C}_n(a,\delta/4)^c
\bigg) \\
&\quad+
\mathbb{P}\bigg(
\frac{|\partial^{*}\Bnb|}{|\Bnb|}
\ge \ct,
\mathcal{C}_n(a,\delta/4),
\overset{\circ}{J_a}(\en)
\nsubseteq \partial^{*}\Bnb
\bigg).
\end{align*}
By Proposition~\ref{proba_B_goes_to_1} and Proposition~\ref{J_lies_on_hole_hull},
there exists $a \ge 0$ such that, for $n$ large enough, the sum of the two right terms is at most $\eta$. Combining the last facts and using Markov's inequality, we obtain
\begin{align*}
\mathbb{P}\bigg(\frac{|\partial^{*}\Bnb|}{|\Bnb|}
\ge \ct
,
\;
|\partial^{*}\Bnb|
\le (1-4\delta)m_{\theta}^{-1}|\mathbf{p}^{n}|
\bigg) \le
\delta^{-1} \eta.
\end{align*}
Letting $\eta\to 0$ concludes the proof.
\end{proof}

\section{Growth of hulls}\label{section_growth_of_hulls}
This section is dedicated to proving Theorem~\ref{growth_of_balls}. We will use the notation $B_r^{\bullet}$ instead of $B_r^{\bullet}(\Tng,e_n)$ when this is not ambiguous. Proposition~\ref{hull_well_defined} shows that understanding $\log(|B_r^{\bullet}|)$ is closely related to understanding $\log(|\partial^{*}B_r^{\bullet}|)$. Our main task is therefore to estimate the law of $\frac{|\partial^{*}B_{r+1}^{\bullet}|}{|\partial^{*}B_{r}^{\bullet}|}$ as long as $|\partial^{*}B_r^{\bullet}| \le n^{1-\varepsilon}$. For $\varepsilon > 0$, we recall that the event $\mathbf{Iso}_n(\varepsilon)$ occurs when the conclusions of Theorem~\ref{isoperimetric_inequalities} and Proposition~\ref{multicurves_few_boundaries} hold. For $\varepsilon,\eta,\beta > 0$, we recall that the event $\mathbf{Iso}_n(\varepsilon,\eta,\beta)$ occurs when $\mathbf{Iso}_n(\varepsilon)$ occurs and when $\mathrm{Isol}_{\beta}(\Tng) \le \eta n$. We now outline the proof. 

\begin{itemize}
	\item[$\bullet$] We fix $\varepsilon,\eta > 0$ and $\beta > 0$ such that, for $n$ large enough, the event $\mathbf{Iso}_n(\varepsilon-2\eta,\eta,\beta)$ occurs with high probability. Hence, for the remainder of the proof, we work under this event.
	
	\item[$\bullet$] We begin by proving that it is unlikely for the volume of the hulls to experience a ``large jump''. More precisely, with high probability, for any $0 \le r \le (1-\varepsilon)D_{\theta}\log(n)$, if $|B_r^{\bullet}| \le n^{1-\varepsilon+\eta}$, then $|B_{r+1}^{\bullet}| \le n^{1-\varepsilon+2\eta}$. This is done in Proposition~\ref{no_big_jumps}.
	
	\item[$\bullet$] Using these facts, the problem reduces to proving that, with high probability,
	\[
	\somme{r=0}{(1-\varepsilon)D_{\theta}\log(n)-1}{\bigg|\log\bigg(\frac{|\partial^{*}B_{r+1}^{\bullet}|}{|\partial^{*}B_{r}^{\bullet}|}\bigg)-D_{\theta}^{-1}\bigg|} = o(\log(n)).
	\]
	We will bound the terms for $ r < \log(n)^{\frac{1}{2}}$ using a crude argument inspired by that of \cite[Proposition 9]{budzinski2023distancesisoperimetricinequalitiesrandom} and the terms for $r \ge \log(n)^{\frac{1}{2}} $	using the $L^1$ convergence obtained in Proposition~\ref{growth_rate_first_moment}.
\end{itemize}

 Let $e_n$ denote a uniformly chosen oriented edge of $\Tng$ and let $f_n$ be the triangle to its right. Then $f_n$ is uniformly distributed among the faces of $\Tng$. For any $0 < \eta < \frac{\varepsilon}{2}$ and $\beta > 0$, we define the event \[
\AnII = \mathbf{Iso}_n\!\Big(\tfrac{\varepsilon}{2}-\eta\Big)\cap\big\{\mathrm{Isol}_{\beta}(\Tng)\le \eta n\big\}\cap\big\{f_n \text{ is not } \beta\text{-isolated}\big\}.
\]

 We claim that, for any $\varepsilon,\eta, \delta > 0$ with $\delta < \eta$, there exists $\beta > 0$ such that for $n$ large enough
 \begin{align}\label{Antypical}
 \mathbb{P}(\AnII) \ge 1- \delta. 
 \end{align}

 Indeed, by Corollary~\ref{Iso_typical}, $\mathbb{P}\big(\mathbf{Iso}_n(\tfrac{\varepsilon}{2}-\eta)\big)\to 1$, so for $n$ large enough $\mathbb{P}\big(\mathbf{Iso}_n(\tfrac{\varepsilon}{2}-\eta)^{c}\big)\le \tfrac{\delta}{4}$. Second, applying Proposition~\ref{few_isolated_faces} with parameter $\tfrac{\delta}{4}$, we fix $\beta>0$ such that, for $n$ large enough, $\mathbb{P}\big(\mathrm{Isol}_{\beta}(\Tng)\ge \tfrac{\delta}{4} n\big)\le \tfrac{\delta}{4}$. Since $\tfrac{\delta}{4}<\eta$, this gives in particular $\mathbb{P}\big(\mathrm{Isol}_{\beta}(\Tng) > \eta n\big)\le \tfrac{\delta}{4}$. Third, since $f_n$ is uniformly distributed among the $2n$ faces of $\Tng$, we have $\mathbb{P}\big(f_n \text{ is } \beta\text{-isolated}\mid \Tng\big)=\mathrm{Isol}_{\beta}(\Tng)/(2n)$; splitting according to whether $\mathrm{Isol}_{\beta}(\Tng)<\tfrac{\delta}{4}n$ or not and using $\mathrm{Isol}_{\beta}(\Tng)\le 2n$,
\[
\mathbb{P}\big(f_n \text{ is } \beta\text{-isolated}\big)=\mathbb{E}\Big[\tfrac{\mathrm{Isol}_{\beta}(\Tng)}{2n}\Big]\le \tfrac{\delta}{8}+\mathbb{P}\big(\mathrm{Isol}_{\beta}(\Tng)\ge \tfrac{\delta}{4}n\big)\le \tfrac{3\delta}{8}.
\]
A union bound over these three events yields, for $n$ large enough, $\mathbb{P}(\AnII^{c})\le \tfrac{\delta}{4}+\tfrac{\delta}{4}+\tfrac{3\delta}{8}=\tfrac{7\delta}{8}<\delta$, which proves \eqref{Antypical}.

We recall the constants $\ct,K_{\theta} \ge 0$ defined in Theorem~\ref{isoperimetric_inequalities}.  Under $\AnII$, as long as $r$ satisfies $ |B_r^{\bullet}| \le n^{1-\varepsilon}$, we have
\[
|\partial^{*}B_r^{\bullet}| \ge \beta|B_r^{\bullet}| \ge \frac{\beta}{3} |\partial^{*}B_{r-1}^{\bullet}|.
\]

Since each edge of $\partial^{*}B_r^{\bullet}$ is incident with at least one face of $B_{r+1}^{\bullet}\setminus B_r^{\bullet}$, and each face has $3$ edges, the number of faces of $B_{r+1}^{\bullet}\setminus B_r^{\bullet}$ is at least $|\partial^{*}B_r^{\bullet}|/3$. Thus, we have the inequality $|B_{r+1}^{\bullet}|\ge |B_r^{\bullet}|+\frac{|\partial^{*}B_r^{\bullet}|}{3}$. We have in particular
\begin{align}\label{lower_bound_Br}
|B_r^{\bullet}| \ge \bigg(1+\frac{\beta}{3}\bigg)^r.
\end{align}

In particular, for $n$ large enough, we have $|B_{\log(n)^{\frac{1}{2}}}^{\bullet}| \ge K_{\theta}\log(n)$.

  \begin{proposition}\label{no_big_jumps}
 For any $\varepsilon,\eta > 0$, we have 
 \begin{align}\label{no_jump}
\mathbb{P}\bigg(\exists r \in \{\log(n)^{\frac{1}{2}},\ldots,(1-\varepsilon)D_{\theta}\log(n)-1\}:\ |B_{r}^{\bullet}| \le n^{1-\varepsilon+\eta},\ |B_{r+1}^{\bullet}| \ge n^{1-\varepsilon+2\eta}\bigg) \underset{n \to +\infty}{\longrightarrow}0.
 \end{align}
 \end{proposition}

\begin{proof}
Let us fix $\varepsilon,\eta > 0$. Let us also fix $\delta > 0$. By \eqref{Antypical}, let us choose $\beta > 0$ such that
\begin{align*}
\mathbb{P}(\AnII) \ge 1 -\delta.
\end{align*} 
Then, we can bound \eqref{no_jump} by
\begin{align}\label{tobound1}
\mathbb{P}(\AnII^{c}) + \somme{r=\log(n)^{\frac{1}{2}}}{(1-\varepsilon)D_{\theta}\log(n)-1}{\mathbb{P}\big(\AnII, |B_{r}^{\bullet}| \le n^{1-\varepsilon+\eta}, |B_{r+1}^{\bullet}| \ge n^{1-\varepsilon+2\eta}\big)}.
\end{align}

We bound the term in the sum uniformly over $r\ge 0$. Fix $\log(n)^{\frac{1}{2}} \le r \le (1-\varepsilon)D_{\theta}\log(n)-1$. Then
\begin{align}\label{growth_too_fast}
\begin{split}
\mathbb{P}\big(\AnII, |B_{r}^{\bullet}| \le n^{1-\varepsilon+\eta}, |B_{r+1}^{\bullet}| \ge n^{1-\varepsilon+2\eta}\big)
\\= \mathbb{E}\bigg[
\indi{|B_{r}^{\bullet}| \le n^{1-\varepsilon+\eta}}
\mathbb{E}\bigg[
\indi{\AnII}
\indi{|B_{r+1}^{\bullet}| \ge n^{1-\varepsilon+2\eta}}
\ \bigg|\ B_{r}^{\bullet}
\bigg]
\bigg].
\end{split}
\end{align}

Conditionally on $\{|B_{r}^{\bullet}| \le n^{1-\varepsilon+\eta}\}$, set $T_n(r):=\Tng\setminus B_{r}^{\bullet}$ and denote by $\ell_n(r)$ its number of boundary components. Then
$T_n(r)\in\mathcal{T}_{n(r),g_n(r),\mathbf{p}^{n}(r)}$ with
\begin{align}\label{parameters}
\left\{
\begin{array}{l}
n(r)=n-\frac12|B_{r}^{\bullet}|+\frac12|\mathbf{p}^{n}(r)|-\ell_n(r),\\
g_n \ge g_n(r)\ge g_n-\frac12|B_{r}^{\bullet}|,\\
|\mathbf{p}^{n}(r)|\le 3|B_{r}^{\bullet}|.
\end{array}
\right.
\end{align}

It follows that
\[
\frac{g_n(r)}{n(r)}\underset{n\to\infty}{\longrightarrow}\theta,
\qquad
\mathbf{p}^{n}(r)=o(n),
\]
uniformly over all $r\ge0$ such that $|B_{r}^{\bullet}|\le n^{1-\varepsilon+\eta}$.

Since, \emph{a priori}, we do not know yet that $B_{r+1}^{\bullet} \neq \Tng$, we cannot apply the isoperimetric inequality to $B_{r+1}^{\bullet}$. Thus, we consider the intermediate layer $\Srn := B_1^+(T_n(r),\partial T_n(r))$. This is the family of triangulations with holes obtained by taking all the boundaries of $T_n(r)$ and the triangles of $T_n(r)$ with at least one vertex on $\partial T_n(r)$. If $|B_{r}^{\bullet}| \le n^{1-\varepsilon+\eta}$ and $|B_{r+1}^{\bullet}| \ge n^{1-\varepsilon+2\eta}$, we claim that $|\Srn| \ge  n^{1-\varepsilon+\frac{3}{2}\eta}$. To prove this, one can assume the converse and use the isoperimetric inequalities. Since $r \ge \log(n)^{\frac{1}{2}}$, and $\AnII$ occurs, we have $|B_{r}^{\bullet}\cup \Srn| \ge  K_{\theta}\log(n)$. Moreover, using the fact that $|B_{r+1}^{+}| \le |B_{r}^{\bullet}\cup \Srn| \le 2n^{1-\varepsilon+\frac{3}{2}\eta}$, using Proposition~\ref{hull_well_defined}, we deduce that exactly one hole of $B_{r}^{\bullet}\cup \Srn$ is filled by more than $n$ triangles, i.e. $B_{r+1}^{\bullet} \neq \Tng$. Now, using the isoperimetric inequality, we have $|B_{r+1}^{\bullet}| \le \ct^{-1}|\partial^{*}B_{r+1}^{\bullet}| \le \ct^{-1}|\partial^{*}\Srn| \le 3\ct^{-1}|\Srn| \le 3\ct^{-1}n^{1-\varepsilon+\frac{3}{2}\eta} < n^{1-\varepsilon+2\eta}$. This is a contradiction. 
 It follows that we have $|\Srn| \ge  \displaystyle n^{1-\varepsilon+\frac{3}{2}\eta}$. Thus \eqref{growth_too_fast} is bounded by 
\begin{align*}
\mathbb{E}\bigg[\indi{|B_{r}^{\bullet}| \le n^{1-\varepsilon+\eta}}\mathbb{P}\bigg(|\Srn| \ge   n^{1-\varepsilon+\frac{3}{2}\eta}\text{ }\bigg|\text{ }B_{r}^{\bullet}\bigg)\bigg].
\end{align*}

By Proposition~\ref{second_moment_bounded_degree}, for $n$ large enough we have
\begin{align*}
\mathbb{P}\bigg(|\Srn| \ge   n^{1-\varepsilon+\frac{3}{2}\eta}\text{ }\bigg|\text{ }B_{r}^{\bullet}\bigg) \le C_{\theta}\frac{|\mathbf{p}^{n}(r)|}{n^{1-\varepsilon+\frac{3}{2}\eta}} \le C_{\theta}n^{-\frac{\eta}{2}},
\end{align*}
uniformly over all $r\ge 0$ satisfying $|B_{r}^{\bullet}| \le n^{1-\varepsilon+\eta}$. Substituting this in \eqref{tobound1}, for $n$ large enough, we obtain the bound
\begin{align}\label{bound_Bn}
\delta + (1-\varepsilon)C_{\theta}D_{\theta}\log(n)n^{-\frac{\eta}{2}} \underset{n\to +\infty}{\to}\delta.
\end{align}
We conclude the proof by letting $\delta \to 0$.
\end{proof}
 
   \begin{proposition}\label{control_growth_small_r}
 For any $\delta > 0$, we have 
 \begin{align*}
\mathbb{P}\bigg(
\log |B_{\log(n)^{\frac{1}{2}}}^{\bullet}| \ge  \delta \log(n)\bigg) \underset{n \to +\infty}{\longrightarrow}0.
 \end{align*}
 \end{proposition}
\begin{proof}
Let us fix $\delta > 0$. Let us fix $t\in\mathcal{T}(n,g_n)$ and $x \in B_{\log(n)^{\frac{1}{2}}}^{\bullet}(t)$. Applying the classical root transformation that transforms the root edge $e$ into a $1$-gon, then choosing a path $\gamma$ from $\rho$ to $x$ and cutting along this path produces a triangulation $t'  \in \mathcal{T}_{2k+1}(n+k,g_n)$ with  $0 \le k \le \log(n)^{\frac{1}{2}}$ equipped with an extra root edge $e'$ on the boundary that corresponds to the last edge used by $\gamma$. Moreover the map $(t,x) \mapsto (t',e')$ is injective. We obtain the following bound
\begin{align*}
\mathbb{E}[|B_{\log(n)^{\frac{1}{2}}}^{\bullet}|] &\le \frac{1}{\tau(n,g_n)}\somme{k=0}{\log(n)^{\frac{1}{2}}}{(2k+1)\tau_{2k+1}(n+k,g_n)} \\
\le &  \frac{(2\log(n)^{\frac{1}{2}}+1)}{\tau(n,g_n)}\somme{k=0}{\log(n)^{\frac{1}{2}}}{\tau_{2k+1}(n+k,g_n)}.
\end{align*}  
Finally using Lemma~\ref{bounded_ratio_vertices} uniformly over all $k$ in the sum, for $n$ large enough we obtain the bound 
\begin{align*}
\mathbb{E}[|B_{\log(n)^{\frac{1}{2}}}^{\bullet}|] \le 3 \log(n)C_{\theta}^{\log(n)^{\frac{1}{2}}} = n^{o(1)}.
\end{align*}
We conclude the proof using Markov's inequality.
\end{proof}

   \begin{proposition}\label{control_growth_large_r}
 For any $\varepsilon,\eta > 0$, we have 
 \begin{align*}
\mathbb{P}\bigg(
\somme{r=\log(n)^{1/2}+1}{(1-\varepsilon)D_{\theta}\log(n)-1}{
\indi{|B_{r}^{\bullet}|\le n^{1-\varepsilon+\eta}}\bigg|\log\bigg(\frac{|\partial^{*}B_{r+1}^{\bullet}|}{|\partial^{*}B_{r}^{\bullet}|}\bigg)-D_{\theta}^{-1}\bigg|}
\ge \frac{\eta}{4}\log(n)\bigg) \underset{n \to +\infty}{\longrightarrow}0.
 \end{align*}
 \end{proposition}

\begin{proof}

Let us fix $\varepsilon,\eta > 0$. Let us also fix $\delta > 0$. By \eqref{Antypical}, let us choose $\beta > 0$ such that
\begin{align*}
\mathbb{P}(\AnII) \ge 1 -\delta.
\end{align*} 
Using this and Proposition~\ref{no_big_jumps}, it suffices to bound
\begin{align}\label{tobound3}
 \mathbb{P}\bigg(\AnII,
\somme{r=\log(n)^{1/2}+1}{(1-\varepsilon)D_{\theta}\log(n)-1}{
\indi{|B_{r}^{\bullet}|\le n^{1-\varepsilon+\eta}}
\indi{|B_{r+1}^{\bullet}|\le n^{1-\varepsilon+2\eta}}
\bigg|\log\bigg(\frac{|\partial^{*}B_{r+1}^{\bullet}|}{|\partial^{*}B_{r}^{\bullet}|}\bigg)-D_{\theta}^{-1}\bigg|}
\ge \frac{\eta}{4}\log(n)\bigg).
\end{align}
We bound this with Markov's inequality. This gives 
\begin{align}\label{bound_Dndelta}
\bigg(\frac{\eta}{4}\log(n)\bigg)^{-1}\somme{r=\log(n)^{\frac{1}{2}}+1}{(1-\varepsilon)D_{\theta}\log(n)-1}{\mathbb{E}\bigg[\indi{\AnII}\indi{|B_{r}^{\bullet}|\le n^{1-\varepsilon+\eta}}\indi{|B_{r+1}^{\bullet}|\le n^{1-\varepsilon+2\eta}}\bigg|\log\bigg(\frac{|\partial^{*}B_{r+1}^{\bullet}|}{|\partial^{*}B_{r}^{\bullet}|}\bigg) - D_{\theta}^{-1}\bigg|\bigg]}.
\end{align}
For any $r\ge 0$, using the fact that $ \frac{|\partial^{*}B_{r+1}^{\bullet}|}{|\partial^{*}B_{r}^{\bullet}|}\ge \frac{\beta}{3}$, we find 
\begin{align} \label{ineq_toref}
&\mathbb{E}\bigg[\indi{\AnII}\indi{|B_{r}^{\bullet}|\le n^{1-\varepsilon+\eta}}\indi{|B_{r+1}^{\bullet}|\le n^{1-\varepsilon+2\eta}}\bigg|\log\bigg(\frac{|\partial^{*}B_{r+1}^{\bullet}|}{|\partial^{*}B_{r}^{\bullet}|}\bigg) - D_{\theta}^{-1}\bigg|\bigg] \\&\le \nonumber \lambda \mathbb{E}\bigg[\indi{\AnII}\indi{|B_{r}^{\bullet}|\le n^{1-\varepsilon+\eta}}\indi{|B_{r+1}^{\bullet}|\le n^{1-\varepsilon+2\eta}}\bigg|\frac{|\partial^{*}B_{r+1}^{\bullet}|}{|\partial^{*}B_{r}^{\bullet}|} - m_{\theta}^{-1}\bigg|\bigg],
\end{align}
where $\lambda = \max(m_{\theta},3\beta^{-1})$.

 Under $\AnII$, for any $r \in \{\log(n)^{\frac{1}{2}}+1,\cdots,(1-\varepsilon)D_{\theta}\log(n)-1\}$ such that $|B_{r}^{\bullet}|\le n^{1-\varepsilon+\eta}$, by \eqref{lower_bound_Br}, we have $ |B_{r}^{\bullet}| \ge   K_{\theta}\log(n)$. Conditionally on $\{|B_{r}^{\bullet}| \le n^{1-\varepsilon+\eta}\}$, set $T_n(r):=\Tng\setminus B_{r}^{\bullet}$ and denote by $\ell_n(r)$ its number of boundary components. As in \eqref{parameters}, we have $T_n(r)\in\mathcal{T}_{\mathbf{p}^{n}(r)}(n(r),g_n(r))$ such that $(n(r),g_n(r),\mathbf{p}^{n}(r))_{r\ge 0}$ satisfy
\[
\frac{g_n(r)}{n(r)}\xrightarrow[n\to\infty]{}\theta,
\qquad
\mathbf{p}^{n}(r)=o(n(r)),
\]
uniformly over all $r\ge0$ such that $|B_{r}^{\bullet}|\le n^{1-\varepsilon+\eta}$. Moreover, the triangulation $T_n(r)$ is uniformly distributed in $\mathcal{T}_{n(r),g_n(r),\mathbf{p}^{n}(r)}$. Moreover, we have $\mathbf{p}^n(r) = |\partial^{*}B_{r}^{\bullet}| \ge \ct|B_{r}^{\bullet}|$. Then, recalling that $\ell_n(r)$ denotes the number of holes of $B_{r}^{\bullet}$ and using the fact that under $\AnII$, Proposition~\ref{multicurves_few_boundaries} is satisfied, we have $$\ell_n(r) \le \displaystyle\frac{\mathbf{p}^n(r) }{\log(\log(n))},$$
uniformly over all $r\ge 0$ such that $|B_{r}^{\bullet}|\le n^{1-\varepsilon+\eta}$.

 Also, note that $B_{r+1}^{\bullet} = B_{r}^{\bullet}\cup \Srb$ where $\Srb:= B_1^{\bullet}(T_n(r),\partial T_n(r))$. Thus, we have $\partial^{*}\Srb = \partial^{*}B_{r+1}^{\bullet}$. Moreover, since $K_{\theta}\log(n) \le|B_{r+1}^{\bullet}| \le n^{1-\varepsilon+2\eta}$, we deduce by the isoperimetric inequality that $\frac{|\partial^{*}B_{r+1}^{\bullet}|}{|B_{r+1}^{\bullet}|}\ge \ct$. It follows that the right-hand side of \eqref{ineq_toref} is bounded by 
 \begin{align*}
 &\lambda \mathbb{E}\bigg[\indi{K_{\theta}\log(n)\le|B_{r}^{\bullet}| \le n^{1-\varepsilon+\eta}}\indi{|\mathbf{p}^{n}(r)| \ge K_{\theta}\log(n)}\indi{\ell_n(r) \le \frac{|\mathbf{p}^{n}(r)|}{\log(\log(n))}}\indi{\frac{|\partial^{*}B_{r+1}^{\bullet}|}{|B_{r+1}^{\bullet}|}\ge \ct}\bigg|\frac{|\partial^{*}B_{r+1}^{\bullet}|}{|\mathbf{p}^{n}(r)|} - m_{\theta}^{-1}\bigg|\bigg]\\
 &=\lambda \mathbb{E}\bigg[\indi{K_{\theta}\log(n) \le|B_{r}^{\bullet}| \le n^{1-\varepsilon+\eta}}\indi{|\mathbf{p}^{n}(r)| \ge K_{\theta}\log(n)}\indi{\ell_n(r) \le \frac{|\mathbf{p}^{n}(r)|}{\log(\log(n))}}\mathbb{E}\bigg[\indi{\frac{|\partial^{*}B_{r+1}^{\bullet}|}{|B_{r+1}^{\bullet}|}\ge \ct}\bigg|\frac{|\partial^{*}B_{r+1}^{\bullet}|}{|\mathbf{p}^{n}(r)|} - m_{\theta}^{-1}\bigg|\bigg|B_{r}^{\bullet}\bigg]\bigg].
 \end{align*}
 Now, using Proposition~\ref{growth_rate_first_moment}, we obtain 
\begin{align*}
\mathbb{E}\bigg[\indi{\frac{|\partial^{*}B_{r+1}^{\bullet}|}{|B_{r+1}^{\bullet}|}\ge \ct}\bigg|\frac{|\partial^{*}B_{r+1}^{\bullet}|}{|\mathbf{p}^{n}(r)|} - m_{\theta}^{-1}\bigg|\bigg|B_{r}^{\bullet}\bigg] \underset{n \to +\infty}{\rightarrow} 0,
\end{align*}
where the convergence holds uniformly over all $r\ge 0$ such that $K_{\theta}\log(n)\le|B_{r}^{\bullet}| \le n^{1-\varepsilon+\eta}$ and $|\mathbf{p}^{n}(r)| \ge K_{\theta}\log(n)$ and $\ell_n(r) \le \frac{|\mathbf{p}^{n}(r)|}{\log(\log(n))}$.
It follows that \eqref{tobound3} is bounded by
\begin{align}
\delta +\bigg(\frac{\eta}{4}\log(n)\bigg)^{-1}\lambda D_{\theta}\log(n) \times o(1) \underset{n \to +\infty}{\rightarrow}  \delta.
\end{align}
 We conclude by letting $\delta \to 0$.
\end{proof}
We can now conclude the proof of Theorem~\ref{growth_of_balls}.
\begin{proof}
Fix $0<\eta<\frac{\varepsilon}{2}$. By Proposition~\ref{control_growth_small_r}, it suffices to prove 
\begin{align*}
\mathbb{P}\bigg(\bigg| \log\bigg(\frac{|B_{(1-\varepsilon)D_{\theta}\log(n)}^{\bullet}|}{|B_{\log(n)^{\frac{1}{2}}}^{\bullet}|}\bigg) - (1-\varepsilon)\log(n) \bigg| \ge \eta\log(n) \bigg) \underset{n \to +\infty}{\longrightarrow}0.
\end{align*}
Let us assume that for any $r \in \{\log(n)^{\frac{1}{2}},\ldots,(1-\varepsilon)D_{\theta}\log(n)-1\}$, if $ |B_{r}^{\bullet}| \le n^{1-\varepsilon+\eta}$ then $|B_{r+1}^{\bullet}| \le n^{1-\varepsilon+2\eta}$ (which occurs with high probability using Proposition~\ref{no_big_jumps}), that $|B_{\log(n)^{1/2}}^{\bullet}|\le n^{\eta/8}$ (which occurs with high probability by Proposition~\ref{control_growth_small_r}) and that the event appearing in the above equation occurs. Then, if $|B_{(1-\varepsilon)D_{\theta}\log(n)}^{\bullet}| \le n^{1-\varepsilon-\eta}$ and since $|B_{(1-\varepsilon)D_{\theta}\log(n)}^{\bullet}| \ge |B_{\log(n)^{\frac{1}{2}}}^{\bullet}|\ge K_{\theta}\log(n)$, using Proposition~\ref{hull_well_defined}, we have for $n$ large enough
\begin{align}\label{approximate_volume_with_perimeter}
\somme{r=\log(n)^{\frac{1}{2}}}{(1-\varepsilon)D_{\theta}\log(n)-1}{\indi{|B_{r}^{\bullet}| \le n^{1-\varepsilon+\eta}}\bigg|\log\bigg(\frac{|\partial^{*}B_{r+1}^{\bullet}|}{|\partial^{*}B_{r}^{\bullet}|}\bigg)-D_{\theta}^{-1}\bigg|} \ge \frac{\eta}{4}\log(n),
\end{align}
where the indicators are all equal to $1$ in this case.

Now, in the case where $|B_{(1-\varepsilon)D_{\theta}\log(n)}^{\bullet}| \ge n^{1-\varepsilon+\eta}$, choosing $\log(n)^{\frac{1}{2}} \le r_0 \le (1-\varepsilon)D_{\theta}\log(n) $ minimal such that $|B_{r_0}^{\bullet}| \ge n^{1-\varepsilon+\frac{3}{4}\eta}$, we deduce the same way that for $n$ large enough
\begin{align*}
\somme{r=\log(n)^{\frac{1}{2}}}{r_0-1}{\indi{|B_{r}^{\bullet}| \le n^{1-\varepsilon+\eta}}\bigg|\log\bigg(\frac{|\partial^{*}B_{r+1}^{\bullet}|}{|\partial^{*}B_{r}^{\bullet}|}\bigg)-D_{\theta}^{-1}\bigg|} \ge \frac{\eta}{4}\log(n),
\end{align*}
which implies \eqref{approximate_volume_with_perimeter}. We conclude the proof using Proposition~\ref{control_growth_large_r} which says that the probability of \eqref{approximate_volume_with_perimeter} tends to $0$ as $n \to +\infty$.
\end{proof}

\printbibliography
\end{document}